\documentclass[11pt]{article}
\usepackage{amsfonts}
\usepackage{amscd}
\usepackage{amsmath}
\usepackage{epsfig}

\def\be{\begin{equation}}
\def\ee{\end{equation}}
\def\ben{\begin{displaymath}}
\def\een{\end{displaymath}}
\def\baa{\begin{eqnarray}}
\def\eaa{\end{eqnarray}}

\def\ba{\begin{array}}
\def\ea{\end{array}}
\makeatletter
\@addtoreset{equation}{section}
\makeatother

\usepackage{amsthm}
\usepackage{amsmath}
\usepackage{amssymb}
\usepackage{amsbsy}
\usepackage{amsfonts}
\usepackage{latexsym}
\usepackage{amscd}
\usepackage{eucal}
\usepackage{mathrsfs}
\usepackage[all, knot]{xy}
\xyoption{arc} \xyoption{color}\xyoption{line}
\newtheorem{thm}{Theorem}
\newtheorem{prop}{Proposition}
\newtheorem{lem}{Lemma}
\newtheorem{cor}{Corollary}
\newtheorem{Remark}{Remark}
\newcommand{\Tr}{\operatorname{Tr}}

\begin{document}
\title{Spectral Determinants on Mandelstam Diagrams}

\author{Luc Hillairet, Victor Kalvin, Alexey Kokotov}
\maketitle

\vskip2cm
{\bf Abstract.} 
 We study the regularized determinant of the Laplacian  as a functional on the space of Mandelstam diagrams (noncompact translation surfaces glued from finite and semi-infinite cylinders). A Mandelstam diagram can be considered as a compact Riemann surface equipped with a conformal flat singular metric $|\omega|^2$, where $\omega$ is a meromorphic one-form  with simple poles  such that all its periods
are pure imaginary and all its residues are real. The main result is an explicit formula for the determinant of the Laplacian in terms of the basic objects on the underlying Riemann surface (the prime form, theta-functions, canonical meromorphic bidifferential) and the divisor of the meromorphic form $\omega$. As an important intermediate result we prove a decomposition formula  of the type of Burghelea-Friedlander-Kappeler for the determinant of the Laplacian for flat surfaces with cylindrical ends and conical singularities.

\vskip2cm

\section{Introduction}

Formally, a (planar) Mandelstam diagram   is a strip  $\Pi=\{z\in{\mathbb C}: 0\leq\Im z\leq H\}$  with finite number of slits parallel to the real line. These slits are either finite segments or half-lines, the sides of different slits and  parts of the boundary of the strip are glued together according to a certain gluing scheme.  This gives a surface made from a finite number of finite and semi-infinite cylinders. In addition, the diagram could be twisted via cutting the finite ("interior") cylinders into two  parts by  vertical cuts and gluing these parts back with certain twists; see, e. g.,  \cite{GidWol}, \cite{Gidd} for more details, explanation of the terminology  and proper references to the original physical literature.

One thus obtains a noncompact translation surface ${\cal M}$ or, more precisely, a  flat surface with trivial holonomy with conical singularities
(at the end points of the slits) and cylindrical ends.

One can also consider ${\cal M}$ as a {\it compact} Riemann surface (i. e. an algebraic curve) with flat conformal metric $|\omega|^2$,
where $\omega$ is the meromorphic differential on ${\cal M}$ obtained from the $1$-form $dz$ in a small neighborhood of a nonsingular point of
${\cal M}$
via parallel transport. The differential $\omega$ has zeros at the end points of the slits and the first order poles at the points of
infinity of  cylindrical ends. All the periods of $\omega$ are pure imaginary, all the residues at the poles of $\omega$ are real.

Moving in the opposite direction, one can get a Mandelstam diagram from a Riemann surface and a meromorphic differential with pure imaginary periods and  simple poles
 with real residues. More precisely, let $X$ be a compact Riemann surface of genus $g$ with $n\geq 2$
 marked points $P_1, \dots, P_n$ and let $\alpha_1, \dots, \alpha_n$ be nonzero real numbers such that
$\alpha_1+\dots +\alpha_n=0$. Then there exists a unique meromorphic differential $\omega$ on $X$ with simple poles at $P_1, \dots, P_n$ such that all the periods
of $\omega$ all pure imaginary and ${\rm Res}(\omega, P_k)=\alpha_k$, $k=1, \dots, n$. Moreover, to such a pair $(X, \omega)$ there corresponds a
 Mandelstam diagram (with $n$ semi-infinite cylinders) (see \cite{GidWol}).

The space of Mandelstam diagrams with fixed residues $\alpha_1, \dots, \alpha_n$ (i. e. with fixed circumferences, $|O_1|, \dots, |O_n|$ of the cylindrical ends) is
 coordinatized by the circumferences, $h_i$, of the interior cylinders;  the interaction times $\tau_j$ (see \cite{GidWol} for explanation
  of the terminology) -- the real parts of the $z$-coordinates of the zeros of the differential $\omega$ (we assume that the smallest interaction
  time, $\tau_0$, is equal to $0$, this can be achieved using a horizontal shift of the diagram) and the twist angles $\theta_k$.

 Mandelstam diagrams (with fixed residues $\alpha_1, \dots, \alpha_n$) give a cell decomposition of the moduli space $M_{g, n}$ of compact Riemann
 surfaces of genus $g$ with n marked points. The top-dimensional cell is given by the set  ${\frak S}_{g, n}$ of {\it simple} Mandelstam diagrams,
 for these diagrams the corresponding meromorphic differential $\omega$ has only simple zeros. The parameters
 \begin{equation}\label{coord} h_i, \ \ i=1, \dots, g;\ \ \tau_j, \ j=1, \dots 2g+n-3; \ \ \theta_k, \ \ k=1, \dots, 3g+n-3\end{equation}
 give global coordinates on ${\frak S}_{g, n}$, see Fig.~\ref{fig1}(taken from \cite{Gidd}, p. 93)
  for the case $g=2, n=4$, three poles with negative
  residues, one pole with positive residue.
\begin{figure}
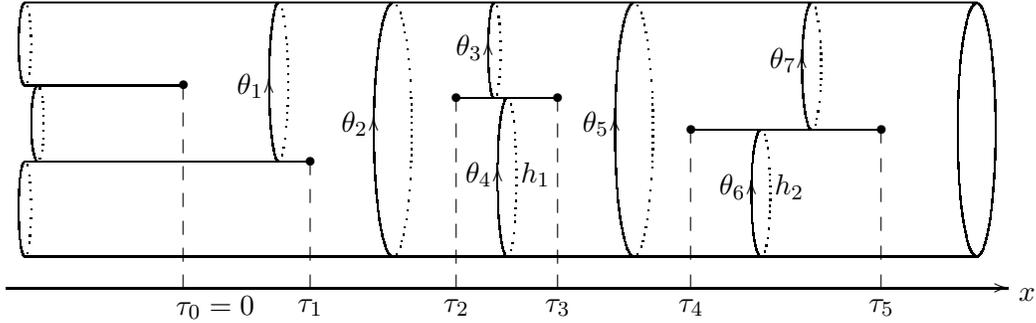
\centering
\[\xy0;/r.20pc/: (33,50)*{\xy
  {\ar@{-}(-70,-20)*{}; (80,-20)*{}}; {\ar@{-}(-70,20)*{}; (80,20)*{}};
{\ar@{->} (-73,-25)*{}; (85,-25)*{}}; (88,-26)*{x};
(40,0)*\ellipse(3,20){-};
(-6,0)*\ellipse(3,20){.};
(-6,0)*\ellipse(3,20)__,=:a(-180){-};
(13,0)*\ellipse(3,20){.};
(13,0)*\ellipse(3,20)__,=:a(-180){-};{\ar@{->}(23,0);(23,2)};(20,1)*{\theta_5};{\ar@{->}(-15,0);(-15,2)};(-18,1)*{\theta_2};
{\ar@{-}(35,0);(65,0)};(35,0)*{ {\scriptstyle\bullet}}; (65,0)*{{\scriptstyle\bullet}};
(27,5)*\ellipse(1.5,10){.};
(27,5)*\ellipse(1.5,10)__,=:a(-180){-};
(23,-5)*\ellipse(1.5,10){.};
(23,-5)*\ellipse(1.5,10)__,=:a(-180){-};
{\ar@{--}(35,0);(35,-25)};(35,-28)*{\tau_4}; {\ar@{--}(65,0);(65,-25)};(65,-28)*{\tau_5};
{\ar@{->}(44.5,-10);(44.5,-8)};(41.5,-9)*{\theta_6};(50.5,-9)*{h_2};{\ar@{->}(52.5,10);(52.5,12)};(49.5,11)*{\theta_7};
{\ar@{-}(-2,5);(14,5)};(-2,5)*{ {\scriptstyle\bullet}}; (14,5)*{{\scriptstyle\bullet}};
(2,6.25)*\ellipse(1,7.5){.};
(2,6.25)*\ellipse(1,7.5)__,=:a(-180){-};
(3,-3.75)*\ellipse(1.5,12.5){.};
(3,-3.75)*\ellipse(1.5,12.5)__,=:a(-180){-};
{\ar@{--}(-2,5);(-2,-25)};(-2,-28)*{\tau_2}; {\ar@{--}(14,5);(14,-25)};(14,-28)*{\tau_3};
{\ar@{->}(3,12);(3,14)};(0,13)*{\theta_3};{\ar@{->}(4.5,-8);(4.5,-6)};(1.5,-7)*{\theta_4} ;(10.5,-7)*{h_1};
{\ar@{--}(-25,-5);(-25,-25)};(-25,-28)*{\tau_1};{\ar@{-}(-70,-5);(-25,-5)};(-25,-5)*{{\scriptstyle\bullet}};
(-15,3.75)*\ellipse(1.5,12.5){.};
(-15,3.75)*\ellipse(1.5,12.5)__,=:a(-180){-};{\ar@{->}(-31.5,6);(-31.5,8)};(-34.5,7)*{\theta_1};
{\ar@{--}(-45,7);(-45,-25)};(-40,-28)*{\tau_0=0};{\ar@{-}(-70,7);(-45,7)};(-45,7)*{{\scriptstyle\bullet}};
(-35,-6.25)*\ellipse(1,7.5){.};
(-35,-6.25)*\ellipse(1,7.5)__,=:a(-180){-};
(-34,0.5)*\ellipse(1,6){.};
(-34,0.5)*\ellipse(1,6)__,=:a(-180){-};
(-35,6.75)*\ellipse(1,6.5){.};
(-35,6.75)*\ellipse(1,6.5)__,=:a(-180){-};
\endxy};
\endxy\]
\caption{Mandelstam Diagram.}\label{fig1}
\end{figure}

  From now on we refer to the coordinates (\ref{coord}) as {\it moduli}.

The goal of the present paper is to study the regularized determinant of the Laplacian on a noncompact translation surface ${\cal M}$ as a function
of moduli (for simplicity we consider only the top dimensional cell).  The title of the paper is chosen to
emphasize the relation to older paper \cite{DPh} (see also \cite{Sonoda}) where
such a determinant was defined in a heuristic way (unrelated to the spectral theory) and the question of the possibility of a spectral definition was
 raised. It should be said that in contrast to \cite{DPh} we are working here with scalar Laplacians, the Laplacians acting on spinors will be
 considered elsewhere.

The scheme of the work can be briefly explained as follows. Assume for simplicity that a Mandelstam diagram  ${\cal M}$ has two cylindrical ends.
Then the Laplacian  $\Delta$ on ${\cal M}$ can be considered as a perturbation of the "free" Laplacian $\mathring\Delta$ on the flat infinite cylinder
$S^1(\frac{H}{2\pi})\times {\mathbb R}$ obtained from the strip $\Pi$ via identifying the points $x\in {\mathbb R}$ with points $x+iH\in {\mathbb R}+iH$.
Then, following the well-known idea (see, e. g., \cite{Mueller}, \cite{ZelHas}, \cite{Carron}), one can introduce the relative operator zeta-function
\begin{equation}\label{zz}
\zeta(s;\Delta, \mathring\Delta)=\frac{1}{\Gamma(s)}\int_0^\infty{\Tr}(e^{-t\Delta}-e^{-t\mathring\Delta})t^{s-1}\,dt
\end{equation}
and define the relative zeta-regularized determinant of the operator $\Delta$ (having continuous spectrum, possibly with embedded eigenvalues - see an example in Appendix A.2)  via
\begin{equation}
\det(\Delta, \mathring\Delta):=e^{-\zeta'(0;\Delta, \mathring\Delta)}\,.\end{equation}

In case of $n\geq 3$ cylindrical ends the definition of $\det(\Delta, \mathring\Delta)$
is similar: as the free Laplacian $\mathring\Delta$ one takes the Laplacian
on the diagram with $n$ semi-infinite slits starting at $\tau_0=0$ (a sphere with $n$ cylindrical ends).

Our main result is an explicit formula for $\det(\Delta, \mathring\Delta)$  in terms of classical objects on the Riemann surface ${\cal M}$ (theta-functions, the prime form, the Bergman kernel) and the divisor of the meromorphic differential $\omega$.

At the first step we establish variational formulas for $\log \det(\Delta, \mathring\Delta)$ with respect to moduli.
At the second step we integrate the resulting system of PDE and get an explicit expression for $\det(\Delta, \mathring\Delta)$ (up to moduli independent constant).
The derivation of the above mentioned variational formulas goes as follows.

First we prove a decomposition formula of the Burghelea-Friedlander-Kappeler type for $\det(\Delta, \mathring\Delta)$. This formula could not be considered as a completely new one: for smooth manifolds with cylindrical ends analogous formulas were found in \cite{LoyaPark} and \cite{MulMul}.
  We believe that it would be possible to
 establish our result just following the way of~\cite{LoyaPark} or~\cite{MulMul} with
 suitable  modifications: presence of conical points  and slightly different method of regularization (in \cite{LoyaPark} and \cite{MulMul}
 the authors use the operators of Dirichlet problem in semi-cylinders
 as "free", whereas we are using here for that purpose the Laplace
 operator in the infinite cylinder), should not present a serious
 additional difficulty. The proofs from \cite{MulMul} are based on
 the one hand on results from scattering theory on manifolds with cylindrical ends
 that themselves depend on techniques of Melrose~\cite{Melrose} and
 on the other hand from results of Carron~ \cite{Carron}.

We have chosen here a slightly different approach that avoids
the full machinery of scattering theory on manifolds with cylindrical
ends (see \cite{Melrose, Christiansen, ChristiansenZworski}). Following \cite{Carron}, it is fairly
straightforward to get a gluing formula for non-zero values of the
spectral parameter so that the only missing ingredient is a precise
description of the resolvent of the operator $\Delta$ at the bottom of
its continuous spectrum (see Theorem 2 below). The latter can then be
obtained using methods of elliptic boundary value problems.

Using the decomposition formula, we reduce the derivation of the variational formulas for  $\det\,\Delta$
to a simpler case of Laplacians (with discrete spectra) on {\it compact} conical surfaces which are flat everywhere except standard fixed "round" ends. After that, using a certain version of the classical Hadamard formula for the variation of the Green function of a plane domain (see Proposition 2),  we derive the variational formulas for the latter simpler case.

The resulting system of PDE for $\log \det\Delta$ (where the so-called Bergman projective connection is the main ingredient)
is a complete analog of the governing equations for the Bergman tau-functions on the Hurwitz spaces and moduli spaces of holomorphic differentials (\cite{KKHur}, \cite{KKZ}, \cite{KokKor}). Relying on the results obtained in~\cite{KokKor,KKHur}, it is not hard to propose an explicit formula for the solution of this system (its main ingredient, the Bergman tau-function on the space of meromorphic differentials of the third kind, was recently introduced by C. Kalla and D. Korotkin in \cite{KorKalla}).

The proof of the thus conjectured formula is a  direct calculation (similar to that from \cite{KokKor}). For the sake of simplicity we present this calculation for genus one Mandelstam diagrams only.

{\bf Acknowledgement}. The work of AK was supported by NSERC. The work
of L.H is partly supported by the ANR programs METHCHAOS and NOSEVOL.
The authors thank D.~Korotkin for extremely useful discussions. We thank also C.~Kalla and and D.~Korotkin for communicating the results from \cite{KorKalla} long before their publication.

\section{Relative determinant and decomposition formula}
Consider $\mathcal M$ as a noncompact flat surface with cylindrical
ends and conical points at the ends of the slits of $\Pi=\{z\in \Bbb
C:0\leq \Im z\leq 1\}$. We shall use $x=\Re z$ as a (global)
coordinate on $\mathcal M$. Let $\mathcal P$ be the set of all conical
points on $\mathcal M$. Assume that  $R>0$ is so large that there are no points in  $\mathcal
P$ with coordinate $x\notin (-R,R)$. Denote $\Gamma=\{p\in\mathcal M: |x|=R\}$ and consider  the (positive) selfadjoint Friedrichs extension  $\Delta^D_{in}$ of the Dirichlet Laplacian on $\mathcal M_{in}=\{p\in\mathcal M: |x|\leq R\}$ with Dirihlet conditions on $\Gamma$. Then for any $f\in C^\infty (\Gamma)$ the  Dirichlet problem
$$
\Delta u=0\text{ on }\mathcal M\setminus\Gamma,\quad  u=f \text{ on }\Gamma$$ has a unique bounded at infinity solution $u$
such that
$u= \tilde f- (\Delta^D_{in})^{-1}\Delta\tilde f$  on $\mathcal M_{in}$,
where $\tilde f\in C^\infty(\mathcal M_{in}\setminus \mathcal P)$ is an extension of $f$.
Introduce the Dirichlet-to-Neumann  operator
$$
\mathcal N f=\lim_{x\to R+}\bigl\langle-\partial_x u (-x),\partial_x u (x) \bigr\rangle+\lim_{x\to R-}\bigl\langle\partial_x u (-x),- \partial_x u (x)\bigr\rangle,
$$
where $\langle\cdot,\cdot\rangle\in L^2(\Gamma_-)\oplus L^2(\Gamma_+)\equiv L^2(\Gamma)$ with $\Gamma_\pm=\{p\in\mathcal M: x=\pm R\}$. The operator $\mathcal N$ is a first order elliptic operator on $\Gamma$ which has zero as an eigenvalue.  The
 modified zeta regularized determinant $\det\!^*\, \mathcal N$
(i.e. the zeta regularized determinant
with zero eigenvalue excluded) is well-defined. By $\det \Delta^D_{in}$ we denote the zeta regularized determinant of $\Delta^D_{in}$.  In this section we prove
\begin{thm}\label{main} The decomposition formula
\begin{equation}\label{MAINone}
\det(\Delta,\mathring\Delta)=C\det\!^* \mathcal N \cdot \det\Delta^D_{in}
\end{equation}
is valid, where $\mathcal N$,
$\Delta^D_{in}$, and  $C$ depend on $R$, however  $C$ is moduli (that is $h_k$,
$\theta_k$, $\tau_k$) independent.
\end{thm}

As it was mentioned in the Introduction this theorem can be considered as a version of the analogous results from \cite{LoyaPark} and \cite{MulMul}
for smooth manifold with cylindrical ends. However, we choose here a different less technical approach that avoids
the full machinery of b-calculus~\cite{Melrose} heavily used in~\cite{LoyaPark} as well as spectral representations and elements of the scattering theory on manifolds with cylindrical ends~\cite{Gu,Do,Mu4} used in~\cite{MulMul}; note that similar spectral representations and results of the scattering theory are also a part of results~\cite{Melrose}  used in~\cite{LoyaPark} (for results of the scattering theory on manifolds with cylindrical ends see also \cite{Christiansen,ChristiansenZworski,Parnovski}).

 As in~\cite{MulMul} our starting point is the  Burghelea-Friedlander-Kappeler type decomposition of $\det(\Delta-\lambda, \mathring\Delta-\lambda)$, obtained in~\cite{Carron} for negative (regular) values of the
spectral parameter $\lambda$. (Although only smooth manifolds are considered in~\cite{Carron},  it is fairly
straightforward to see that on  Mandelstam diagrams the decomposition remains valid outside of conical points.) In order to justify the decomposition for $\det(\Delta, \mathring\Delta)$  (i.e. at  the bottom $\lambda=0$ of the continuous spectrum of $\Delta$ and $\mathring\Delta$), one has to study the behavior of all ingredients of the decomposition formula  as $\lambda\to 0-$ (i.e. zeta regularized determinants of Laplacians and Dirichlet-to-Neumann operators)  and then pass to the limit. Our approach relies on precise information on the behavior of the resolvent of the operator $\Delta$ at $\lambda=0$ (see Theorem 2 below) obtained by well-known methods of elliptic boundary value problems and the Gohberg-Sigal  theory of Fredholm holomorphic functions; see e.g.~\cite{Lions Magenes,KM,KMR} and~\cite{GS} (or e.g.~\cite[Appendix]{KMR}). As a consequence, we immediately get precise information on the behavior of the Dirichlet-to-Neumann operator and an asymptotic of its determinant as $\lambda\to 0$. The latter one also provides the integrand in~\eqref{zz} with  asymptotic as $t\to+\infty$. This together with  asymptotic of the integrand as $t\to 0$ (obtained in a standard way) prescribes the behavior of  $\det(\Delta-\lambda, \mathring\Delta-\lambda)$ as $\lambda\to 0-$ and completes justification of the decomposition formula for  $\det(\Delta, \mathring\Delta)$.

\subsection{Resolvent meromorphic continuation and its singular part at zero}

In this subsection we operate with Friedrichs extensions $\Delta_\epsilon$ of the Lapacian $\Delta$ in weighted spaces $L^2_\epsilon(\mathcal M)$ with different values of weight parameter $\epsilon$. For this reason we reserve the notation $\Delta_0$ for the (selfadjoint nonnegative) Friedrichs extension of the Laplacian
$\Delta$ in $L^2(\mathcal M)$ initially defined on the set
$C_0^\infty(\mathcal M\setminus\mathcal P)$ of smooth compactly
supported functions.

\begin{thm}\label{RSP} By $\mathcal B\bigl( L^2_\epsilon(\mathcal M),
L^2_{-\epsilon}(\mathcal M)\bigr)$, $\epsilon>0$, we denote the
space of bounded operators acting from smaller $
L^2_\epsilon(\mathcal M)$ to bigger $L^2_{-\epsilon}(\mathcal M)$
weighted space with
 the norm $\|u;L^2_{\gamma}(\mathcal
M)\|=\|e_{\gamma} u;L^2(\mathcal M)\|$, where the weight
$e_{\gamma}\in C^\infty(\Bbb R)$ is a positive function such that
$e_{\gamma}(x)=\exp(\gamma|x|)$ for all sufficiently large values of
$|x|$.

Let $\epsilon$ be a sufficiently small positive number. Then the
function
$$\mu \mapsto \bigl(\Delta_0-\mu^2\bigr)^{-1}-\frac{i}{2\mu}\bigl(\cdot,1\bigr)_{L^2(\mathcal M)}\in \mathcal B\bigl(
L^2_\epsilon(\mathcal M), L^2_{-\epsilon}(\mathcal M)\bigr)$$ is holomorphic in the union  $\Bbb C^+_\epsilon$ of $\Bbb C^+=\{\mu\in\Bbb
C:\Im\mu>0\}$ with the disc $|\mu|<\epsilon$. In other words, the
resolvent $(\Delta_0-\mu^2)^{-1}$ viewed as the holomorphic function
$$\Bbb C^+\ni \mu \mapsto (\Delta_0-\mu^2)^{-1}\in \mathcal B\bigl(
L^2_\epsilon(\mathcal M), L^2_{-\epsilon}(\mathcal M)\bigr)$$ has a
meromorphic continuation to $\Bbb C^+_\epsilon$, which is holomorphic
in $\Bbb C^+_\epsilon\setminus\{0\}$ and has a simple pole at zero
with the rank one operator $L^2_{\epsilon}(\mathcal M)\ni f\mapsto
\frac {i}{2}(f, 1)_{L^2(\mathcal M)}\in L^2_{-\epsilon} (\mathcal
M)$ as the residue.
\end{thm}

 The scheme of the proof can be described as follows. We consider the Friedrichs m-sectorial
 extension $\Delta_\epsilon$ of the Laplacian
$\Delta$  initially defined on $C_0^\infty(\mathcal
M\setminus\mathcal P)$ and acting in the weighted space
$L^2_\epsilon (\mathcal M)$. (Here m-sectorial means that the
numerical range $\{ (e_{2\epsilon}\Delta_\epsilon u,
u)_{L^2(\mathcal M)}\in\Bbb C: u\in \mathcal D_\epsilon\}$ and the
spectrum of a closed  operator $\Delta_\epsilon$ in
$L^2_\epsilon(\mathcal M)$ with the domain $\mathcal D_\epsilon$ are
both in some sector $\{\lambda\in\Bbb C:|\arg(\lambda+c)|\leq
\vartheta<\pi/2, c>0\}$.)
 Then we
introduce a certain rank $n$ extension of the operator $\Delta_\epsilon-\mu^2$  ($n$ stands for
the number of cylindrical ends on $\mathcal M$).
The inverse of that extension provides the resolvent
$(\Delta_0-\mu^2)^{-1}$ with the desired meromorphic continuation to
$\Bbb C^+_\epsilon$. The proof of Theorem~\ref{RSP} is preceded by Lemmas~\ref{fr},~\ref{l2} and Proposition~\ref{par}.

In order to introduce  $\Delta_\epsilon$ we need to obtain some
estimates on the quadratic form
$$
\mathsf q_\epsilon[u,u]=\|\nabla u; L^2_\epsilon(\mathcal
M)\|^2+\bigl((\partial_x e_{2\epsilon})(\partial_x u),
u\bigr)_{L^2(\mathcal M)},\quad u\in C_0^\infty(\mathcal
M\setminus\mathcal P),
$$
of the Laplacian $\Delta$ in $L^2_\epsilon(\mathcal M)$.  Denote by
$H^1_\epsilon(\mathcal M)$ the weighted Sobolev space
 of  functions $v=e_{-\epsilon}u$, $u\in H^1(\mathcal
M)$, with the norm $\|v; H^1_\epsilon (\mathcal M)\|=\|e_\epsilon v;
H^1(\mathcal M)\|$; here $H^1(\mathcal M)$ is the completion of the
set $C_0^\infty(\mathcal M\setminus\mathcal P)$ in the norm $$ \|u;
H^1(\mathcal M)\|=\sqrt{\|u;L^2(\mathcal M)\|^2+\|\nabla u;
L^2(\mathcal M)\|^2}.$$ Clearly, $|\partial_x e_{2\epsilon}(x)|\leq
C e_{2\epsilon}(x) $,
$$
\begin{aligned}
|\bigl((\partial_x e_{2\epsilon})(\partial_x u),
u\bigr)_{L^2(\mathcal M)}|&\leq C\|\partial_x u;L^2_\epsilon(\mathcal
M)\|\cdot\|u;L^2_\epsilon(\mathcal M)\|\\
&\leq
C^2\delta^{-1}\|u;L^2_\epsilon(\mathcal M)\|^2+ \delta\|\nabla u;
L^2_\epsilon(\mathcal M)\|^2,\quad \delta>0,
\end{aligned}
$$
and the norm in $H^1_\epsilon(\mathcal M)$ is equivalent to the norm
$ \sqrt{\|u; L^2_\epsilon(\mathcal M)\|^2+\|\nabla u;
L^2_\epsilon(\mathcal M)\|^2}$. Thus for some $\delta>0$ and
$\gamma>0$ we obtain
$$
|\arg (\mathsf q_\epsilon[u,u]+\gamma\|u;L^2_\epsilon(\mathcal
M)\|^2)|\leq \vartheta<\pi/2,
$$
$$
\delta\|u; H^1_\epsilon(\mathcal M)\|^2-\gamma\|u;
L^2_\epsilon(\mathcal M)\|^2\leq \Re \mathsf q_\epsilon[u,u]\leq
 \delta^{-1} \|u; H^1_\epsilon(\mathcal M)\|^2,
$$
which shows that $\mathsf q_\epsilon$ with  domain
$H_\epsilon^1(\mathcal M)$ is a closed densely defined sectorial
form in $L^2_\epsilon(\mathcal M)$. Therefore this form uniquely determines an
m-sectorial operator $\Delta_\epsilon$ (the Friedrichs extension of
the Laplacian $\Delta: C_0^\infty(\mathcal M\setminus \mathcal P)
\to L^2_\epsilon(\mathcal M)$, see~\cite[Theorem VI.2.1]{Kato})
possessing the properties: i) The domain $\mathcal D_\epsilon$  of
$\Delta_\epsilon$ is dense in $H^1_\epsilon(\mathcal M)$; ii) For
all $u\in \mathcal D_\epsilon$ and $v\in H_\epsilon^1(\mathcal M)$
we have $ (e_{2\epsilon} \Delta_\epsilon u, v)_{L^2(\mathcal
M)}=\mathsf q _\epsilon[u,v]$. This extension scheme also gives the
nonnegative selfadjoint Friedrichs extension $\Delta_0$ if we
formally set $\epsilon=0$; the operator $\Delta_\epsilon$ (with
$\epsilon>0$) is non-selfadjoint. Due to conical points on $\mathcal
M$ the second derivatives of $u\in\mathcal D_\epsilon$ are not
necessarily in $L^2_\epsilon(\mathcal M)$, e.g.~\cite{KMR}.


\begin{lem}\label{fr} Equip the domain $\mathcal
D_\epsilon$ of $\Delta_\epsilon$ with the graph norm
\begin{equation}\label{norm}
\|u; \mathcal D_\epsilon \|=\sqrt{\|u; L^2_\epsilon(\mathcal
M)\|^2+\|\Delta_\epsilon u; L^2_\epsilon(\mathcal M)\|^2}.
\end{equation}
Then the continuous operator
\begin{equation}\label{ep}
\Delta_\epsilon-\mu^2: \mathcal D_\epsilon\to L^2_\epsilon(\mathcal
M)
\end{equation} is Fredholm (or, equivalently, $\mu^2$ is not in the essential spectrum of $\Delta_\epsilon$)
 if and only if for any $\xi\in\Bbb R$ the point
 $\mu^2-(\xi+i\epsilon)^2$ is not in the spectrum $\{0,4\pi^2
\ell^2 |O_k|^{-2}:\ell\in\Bbb N, 1\leq k\leq n\}$ of the selfadjoint
Laplacian on the union of circles $O_1,\dots,O_n$. The essential spectrum of $\Delta_\epsilon$ is depicted on Fig.~\ref{fig5}.
\end{lem}
\begin{proof} See Appendix. \end{proof}
\begin{figure}
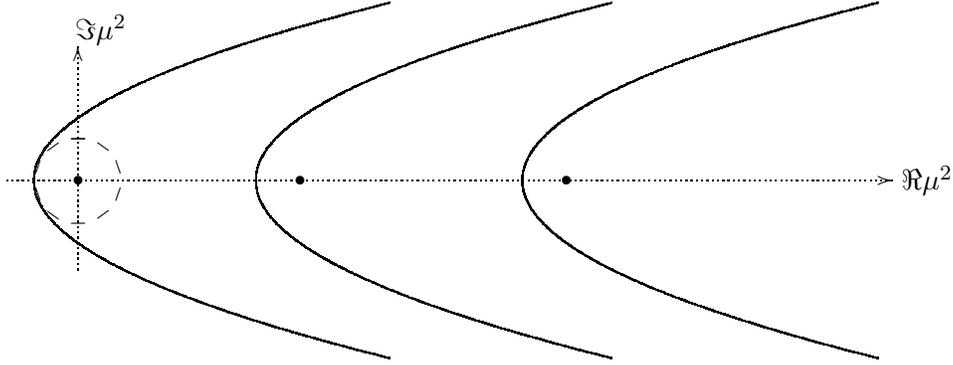
\centering
\[
\xy (0,0)*{\xy0;/r.28pc/:{\ar@{.>}(-3,0);(100,0)*{\ \Re \mu^2}};
{\ar@{.>}(5,-10);(5,15)}; (7,17)*{\ \Im \mu^2};
{(5,0)*{\scriptstyle\bullet}};(40,-20);(40,20)**\crv{(-40,0)};
(5,0)*\xycircle<16pt>{--};
{(30,0)*{\scriptstyle\bullet};}; (65,-20);(65,20)**\crv{(-15,0)};
{(60,0)*{\scriptstyle\bullet}};(95,-20);(95,20)**\crv{(15,0)};
\endxy};
\endxy
\]
\caption{Essential spectrum $\sigma_{ess} (\Delta_\epsilon)$ of the
operator $\Delta_\epsilon$ for
 $\epsilon>0$, where the points marked as $\bullet$ represent eigenvalues of the selfadjoint Laplacian on the union of circles $O_1,\dots,O_n$, and solid lines are parabolas of $\sigma_{ess} (\Delta_\epsilon)$. Dashed line corresponds to the boundary $|\mu|=\epsilon$ of the disc $|\mu|<\epsilon$. 
 As $\epsilon\to 0$ the parabolas collapse to  rays
 forming the essential spectrum $\sigma_{ess} (\Delta)=[0,\infty)$.}\label{fig5}
\end{figure}

\begin{lem}\label{l2} Take some functions
$\Bbb C\ni \mu\mapsto \varphi_k(\mu)\in C^\infty(\mathcal M\setminus\mathcal P)$
 satisfying
\begin{equation}\label{phi}
\varphi_k(\mu;p)=\left\{
              \begin{array}{ll}
                e^{i\mu |x|}, &  p=(x,y)\in(-\infty,-R-1)\times O_k \  (\text {resp. } p\in(R+1,\infty)\times O_k  ), \\
                0, &  p\in\mathcal M\setminus (-\infty,-R)\times O_k \ (\text{resp. } p\in\mathcal M\setminus (R,\infty)\times O_k),
              \end{array}\right.
\end{equation}
if $O_k$ corresponds to a cylindrical end directed to the left
(resp. to the right).  Let $\mu\in\Bbb C^+$ and $|\mu|<\epsilon$, where $\epsilon>0$ is sufficiently small.
Then for any $f\in
L^2_\epsilon(\mathcal M)$ and some
 $c_k\in\Bbb C$, which depend on $\mu$ and $f$,
 we have
\begin{equation}\label{as}
(\Delta_0-\mu^2)^{-1}f-\sum_{k=1}^{n}c_k\varphi_k(\mu)\in\mathcal
D_\epsilon.
\end{equation}
 \end{lem}
\begin{proof} See Appendix. \end{proof}
Clearly,  $(\Delta-\mu^2)\varphi_k(\mu)\in C_0^\infty(\mathcal M\setminus\mathcal P)\subset L^2_\epsilon(\mathcal M)$. Thus Lemma~\ref{as} implies that the linear combinations  of $\varphi_1(\mu),
\dots, \varphi_n(\mu)$  are asymptotics of $(\Delta_0-\mu^2)^{-1}f$
as $|x|\to \infty$  with a remainder in the space $\mathcal
D_\epsilon$ of functions exponentially decaying at infinity. We
introduce a rank $n$ extension $\mathcal A(\mu): \mathcal
D_\epsilon\times \Bbb C^n\to L^2_\epsilon (\mathcal M)$ of the
m-sectorial operator~$\Delta_\epsilon-\mu^2$ by considering the values of $\Delta-\mu^2$ not only on $\mathcal D_\epsilon$ but also on the asymptotics $\sum
c_k\varphi_k(\mu)$. The
continuous operator $\mathcal A(\mu)$ acts by the formula
\begin{equation}\label{A}
\mathcal D_\epsilon\times \Bbb C^n\ni (u,c)\mapsto \mathcal
A(\mu) (u,c)=(\Delta_\epsilon-\mu^2)u +\sum_{k=1}^n c_k
(\Delta-\mu^2)\varphi_k(\mu)\in L^2_\epsilon(\mathcal M).
\end{equation}
 We shall also use the operator $\mathcal J(\mu)$ mapping $\mathcal D_\epsilon\times \Bbb C^n$ into $L^2_{-\epsilon} (\mathcal M)$ in the following natural way:
$$
\mathcal D_\epsilon\times \Bbb C^n\ni (u,c)\mapsto \mathcal J(\mu)(u,c)=u+\sum c_k\varphi_k(\mu)\in L^2_{-\epsilon}(\mathcal M), \quad |\mu|<\epsilon.
$$
The functions $\varphi_k(\mu)$, $1\leq k\leq n$, are linearly independent and for $|\mu|<\epsilon$ we have $\varphi_k(\mu)\notin\mathcal D_\epsilon$. Hence $\mathcal J(\mu)$ yields an isomorphism between $\mathcal D_\epsilon\times \Bbb C^n$ and its range $\{u+\sum c_k \varphi_k(\mu): u\in\mathcal D_\epsilon, c\in\Bbb C^n\}\subset L^2_{-\epsilon}(\mathcal M)$.

Below we will rely on some results of the theory of Fredholm holomorphic functions, see e.g.~\cite{GS} or\cite[Appendix A]{KM}. Recall that holomorphic in a domain $\Omega$ operator function $\mu\mapsto F(\mu)\in\mathcal B (\mathcal X,\mathcal Y)$, where $\mathcal X$ and $\mathcal Y$ are some Banach spaces, is called Fredholm if the operator $F(\mu): \mathcal X\to\mathcal Y$ is Fredholm for all $\mu\in\Omega$ and $F(\mu)$ is invertible for at least one value of $\mu$. The spectrum of a Fredholm holomorphic function $F$ (which is the subset of $\Omega$, where $F(\mu)$ is not invertible) consists of isolated eigenvalues of finite algebraic multiplicity.  Let $\psi_0$ be an eigenvector corresponding to an eigenvalue $\mu_0$ of $F$ (i.e. $\psi_0\in\ker F(\mu_0)\setminus\{0\}$). The elements $\psi_1,\dots\psi_{m-1}$ in $\mathcal X$ are called generalized eigenvectors if they satisfy $\sum_{j=0}^\ell \frac{1}{j!} \partial^j_\mu F(\mu_0)\psi_{\ell-j}=0$, $\ell=1,\dots,m-1$. If there are no generalized eigenvectors and $\dim \ker F(\mu_0)=1$ , we say that $\mu_0$ is a simple eigenvalue of $F$. Let $\mu_0$ be a simple eigenvalue of a Fredholm holomorphic function $F$. Then in a neighborhood of  $\mu_0$ the inverse $F(\mu)^{-1}$ of the operator $F(\mu)$ admits the representation
\begin{equation}\label{SP}
F(\mu)^{-1}=\frac {\omega_0(\cdot)\,\psi_0}{\mu-\mu_0}+H(\mu),
\end{equation}
where $\mu\mapsto H(\mu)\in \mathcal B(\mathcal X,\mathcal Y)$ is holomorphic,  $\psi_0\in\ker F(\mu_0)\setminus\{0\}$, and $\omega_0\in \ker F^*(\overline{\mu_0}) $ is an eigenvector of the adjoint holomorphic function $\mu\mapsto F^*(\mu)= \bigl(F(\overline{\mu})\bigr)^*\in \mathcal B (\mathcal Y^*,\mathcal X^*)$ such that the value of the functional $\omega_0(\cdot )$ on $\partial_\mu F(\mu_0)\psi_0$ is $1$. Note that the converse is also true, i.e.~\eqref{SP} implies that  $\mu_0$ is a simple eigenvalue of $F$ and $\psi_0\in\ker F(\mu_0)$.  For the proof of~\eqref{SP} we refer to~\cite[Theorem 7.1]{GS} and~\cite[Theorem A.10.2]{KM}.


\begin{prop}\label{par} Let $\epsilon>0$ be sufficiently small. Then
\begin{enumerate}
\item $\mu\mapsto\mathcal A(\mu)\in \mathcal B\bigl(\mathcal D_\epsilon\times\Bbb C^n, L^2_\epsilon(\mathcal M)\bigr)$ is a Fredholm holomorphic
operator function  in the disc $|\mu|<\epsilon$ and $\mathcal A(\mu)$ is invertible for all  $\mu\in\Bbb C^+$.
\item $\ker \mathcal A(0)=\{\mathcal J(0)^{-1}C: C\in\Bbb C\}$ and $\ker \mathcal A(0)^*=\{e_{-2\epsilon}C: C\in \Bbb C\}$.
\item There are no solutions $(v, d)$ to the equation $(\partial_\mu \mathcal A(0))  (\mathcal J(0))^{-1} 1+ \mathcal A^\epsilon _0 (v,d) =0$, i.e. there are no generalized eigenvectors and $\mu=0$ is a simple eigenvalue of  $\mathcal A(\mu)$.
\item In the disc $|\mu|<\epsilon$ we have
$$
\mathcal A(\mu)^{-1}=\frac {i}{2\mu}(\cdot , 1)_{L^2(\mathcal M)}\mathcal J(0)^{-1} 1+\mathcal H(\mu),
$$
where $\mu\mapsto \mathcal H(\mu)\in \mathcal B\bigl( L^2_\epsilon(\mathcal M),\mathcal D_\epsilon\times\Bbb C^n\bigr)$ is holomorphic.

\item For $\mu\in\Bbb C^+$ with  $|\mu|<\epsilon$ the operators
$\mathcal J(\mu)\mathcal A(\mu)^{-1}$
and $(\Delta_0-\mu^2)^{-1}$   coincide as elements of
$\mathcal B(L^2_\epsilon(\mathcal M),L^2_{-\epsilon}(\mathcal M))$.
Thus
 $\mathcal J(\mu)\mathcal
A(\mu)^{-1}$ provides the resolvent $$\Bbb C^+\ni \mu
\mapsto (\Delta_0-\mu^2)^{-1}\in \mathcal B\bigl(
L^2_\epsilon(\mathcal M), L^2_{-\epsilon}(\mathcal M)\bigr)$$ with
meromorphic continuation to the disc $|\mu|<\epsilon$.
\end{enumerate}
\end{prop}
\begin{proof} 1. For $|\mu|<\epsilon$ the operator $\mathcal A(\mu)$
is Fredholm  as a finite-rank extension of a Fredholm operator, see
Lemma~\ref{fr}. It is easy to see that for any $(u,c)\in\mathcal
D_\epsilon\times\Bbb C^n$ the function $\mu\mapsto \mathcal
A(\mu)(u,c)$ is holomorphic in the disc $|\mu|<\epsilon$. Assume, in addition, that $\mu\in\Bbb C^+$. Then for any $(u,c)\in \mathcal D_\epsilon\times \Bbb C^n$ we have $\mathcal
J(\mu) (u,c)\in L^2(\mathcal M)$ and $\mu^2$ is a regular
point of the nonnegative selfadjoint operator $\Delta_0$. Hence $\dim\ker \mathcal A(\mu) =0$.
Indeed, for any $(u,c)\in \ker \mathcal A(\mu)$ we have
$(\Delta_0-\mu^2)\mathcal J(\mu)(u,c)=0$, which implies
$\mathcal J(\mu)(u,c)=0$, and therefore $v=0$ and $a=0$.
Besides, by Lemma~\ref{l2} for any $f\in L^2_\epsilon(\mathcal M)$
we have $(\Delta_0-\mu^2)^{-1}f=u+\sum c_k \varphi_k(\mu)$ with
$u\in \mathcal D_\epsilon$ and $c\in\Bbb C^n$. Therefore $\mathcal
A(\mu)(u,c)=f$ and the operator $\mathcal A(\mu)$ is
invertible. Assertion 1 is proved.

 2. It is easy to see that $\mathcal A(0)\mathcal J(0)^{-1}1=0$; here $\mathcal J(0)^{-1}1=(1-\sum\varphi_k(0), 1,\dots, 1)$. Moreover, since there are no other bounded solutions to $\Delta u=0$ than a constant, this proves the first equality in Assertion 2. For $v$ in the kernel of the adjoint operator   $\mathcal A(0)^*: L^2_\epsilon\to \mathcal D_\epsilon^*\times\Bbb C^n$ and any $(u,c)\in \mathcal D_\epsilon\times \Bbb C^n$ we have
 $$
 (\Delta_\epsilon u, v )_{L^2_\epsilon(\mathcal M)}+\sum c_k (\Delta  \varphi_k(0), v )_{L^2_\epsilon(\mathcal M)}=0.
 $$
Therefore $v$ is an element in $\ker (\Delta_\epsilon)^*   \subset \mathcal D_\epsilon$  satisfying $ (\Delta  \varphi_k(0), v )_{L^2_\epsilon(\mathcal M)}=0$,  $1\leq k\leq n$; here  $(\Delta_\epsilon)^*$  is the adjoint to $\Delta_\epsilon$ m-sectorial operator in
$L^2_\epsilon(\mathcal M)$ with domain $\mathcal D_\epsilon$; see Proof of Lemma~\ref{fr} in Appendix. Separation of variables in the cylindrical ends gives
\begin{equation}\label{SV}
 e_{2\epsilon}(x)  v(x,y) = A_k  x +\sum_{\ell\in \Bbb Z } B_k^\ell e^{2\pi |O_k|^{-1}(- |\ell x|+ i\ell y)},
 \end{equation}
where
$A_k$   and $B^\ell_k$ are some coefficients. Next we show that $A_k=0$. Consider, for instance, a right cylindrical end (with  cross-section $O_k$). Then by using the  Green formula we get
$$
\begin{aligned}
0= (\Delta  \varphi_k(0), v )_{L^2_\epsilon(\mathcal M)}=\lim_{T\to+\infty} \int_R^T\int_{O_k} e_{2\epsilon } (x)\Delta  \varphi_k(0;x,y) \overline{ v(x,y)}\, dx\,dy
\\
=\lim_{T\to+\infty} \int_{O_k}\varphi_k(0;T,y)\overline{\partial_x(e_{2\epsilon }   v)(T,y)}\,dy=|O_k|A_k
 \end{aligned}
$$
and hence  $A_k=0$. Similarly one can see that  $A_k=0$ for the left cylindrical ends. This together with~\eqref{SV} implies that  $e_{2\epsilon}  v$ is a bounded solution and thus it is a constant. Assertion 2 is proved.

3. Taking the derivative in~\eqref{A} we obtain
 $\partial_\mu \mathcal A(0) \mathcal J(0)^{-1} 1 =\sum
\Delta \partial_\mu \varphi_k(0)$. The  equation for $(v,d)$ takes the form
$$
\mathcal A(0) (v,d) =-\sum
\Delta \partial_\mu \varphi_k(0).
$$
This equation has no solutions since its right hand side is not orthogonal to $e_{-2\epsilon}\in\ker \mathcal A(0)^*$. Indeed,
$$
\begin{aligned}
\bigl(\Delta \partial_\mu \varphi_k(0), e_{-2\epsilon}\bigr)_{L^2_\epsilon(\mathcal M)}=\lim_{T\to+\infty}\int_R^T\int_{O_k} \Delta\partial_\mu\varphi_k(0;x,y)\, dx\,dy\\=-\lim_{T\to+\infty} \int_{O_k} \partial_x\partial_\mu\varphi_k(0;T,y)\,dy
=-\lim_{T\to+\infty} \int_{O_k} i\,dy=-i|O_k|
\end{aligned}
$$
if $O_k$ corresponds to a right cylindrical end; in the same way one can check that $(\Delta \partial_\mu \varphi_k(0), e_{-2\epsilon})_{L^2_\epsilon(\mathcal M)}=-i|O_k|$ for the left cylindrical ends. Thus $$\sum\bigl(\Delta \partial_\mu \varphi_k(0), e_{-2\epsilon}\bigr)_{L^2_\epsilon(\mathcal M)}=-2i$$ and there are no generalized eigenvectors. Assertion 3 is proved.

4. Assertion 4 is the representation~\eqref{SP} written for $\mathcal A(\mu)^{-1}$ and $\mu_0=0$. Indeed, $\psi_0=\mathcal J(0)^{-1}1$ is an eigenvector of $\mathcal A(\mu)$ at $\mu=0$, and $\omega_0(\cdot )=\frac i 2 (\cdot, 1)_{L^2(\mathcal M)}$ because
$$
 \omega_0\bigl(\partial_\mu F(\mu_0)\psi_0\bigr)=\frac i 2\bigl( \partial_\mu \mathcal A(0) \mathcal J(0)^{-1} 1 , 1\bigr)_{L^2(\mathcal M)}=\frac i 2\sum\bigl( \Delta \partial_\mu \varphi_k(0) ,e_{-2\epsilon}\bigr)_{L_\epsilon^2(\mathcal M)}=1
$$
as we need.

5. For
 $\mu\in\Bbb C^+$ with $|\mu|<\epsilon$ and any $f\in L^2_\epsilon(\mathcal M)$ we have
$(\Delta_0-\mu^2)^{-1}f=\mathcal J(\mu)\mathcal
A(\mu)^{-1}f$, which means that the operators
$(\Delta_0-\mu^2)^{-1}$ and $\mathcal J(\mu)\mathcal
A(\mu)^{-1}$ in $\mathcal B\bigl(
L^2_\epsilon(\mathcal M), L^2_{-\epsilon}(\mathcal M)\bigr)$ are
coincident. Clearly,
$\mu\mapsto \mathcal J(\mu)$ is holomorphic in the disc
$|\mu|<\epsilon$. Thus the meromorphic in the disc $|\mu|<\epsilon$
function $ \mu\mapsto\mathcal J(\mu)\mathcal
A(\mu)^{-1}\in \mathcal B\bigl( L^2_\epsilon(\mathcal
M), L^2_{-\epsilon}(\mathcal M)\bigr) $ provides the resolvent
$(\Delta_0-\mu^2)^{-1}$ with the desired continuation. Assertion 5 is
proved.
\end{proof}

\begin{proof}[Proof of Theorem~\ref{RSP}] As a consequence of Assertions 4 and 5 of Proposition~\ref{par} we have
$$
\begin{aligned}
&(\Delta_0-\mu^2)^{-1}=\mathcal J(\mu)\mathcal
A(\mu)^{-1}= \frac {i} {2\mu}(\cdot,1)_{L^2(\mathcal M)}+ \Phi(\mu),\\
 &\Phi(\mu)=\frac{i}{2\mu}(\cdot,1)_{L^2(\mathcal M)}\bigl(\mathcal J(\mu)-\mathcal J(0)\bigr)\mathcal J(0)^{-1}1+\mathcal J(\mu) \mathcal H(\mu),
 \end{aligned}
$$
where $\mu\mapsto\Phi(\mu)\in\mathcal B(L^2_\epsilon(\mathcal M),L^2_{-\epsilon}(\mathcal M))$ is holomorphic in the disc $|\mu|<\epsilon$. Theorem~\ref{RSP} is proved.
\end{proof}

\subsection{Dirichlet-to-Neumann  operator} 
\label{DTN}

As before we assume that  $R>0$ is  so large that there are no conical points on $\mathcal
M$ with coordinate $x\notin (-R,R)$ and denote $\Gamma=\{p\in\mathcal M: |x|=R\}$. Then for $\mu^2\in \Bbb C\setminus (0,\infty)$ and any $f\in C^\infty (\Gamma)$ there exists a unique bounded at infinity solution to the  Dirichlet problem
\begin{equation}\label{DP}
(\Delta-\mu^2)u(\mu)=0\text{ on } \mathcal M\setminus\Gamma,\quad u(\mu)=f \text{ on } \Gamma,
\end{equation}
such that
\begin{equation}\label{u}
u(\mu)= \tilde f- (\Delta^D_{in}-\mu^2 )^{-1}(\Delta-\mu^2)\tilde f\text{ on }\mathcal M_{in},
\end{equation}
where $\tilde f\in C^\infty(\mathcal M_{in}\setminus \mathcal P)$ is an extension of $f$ and  $\Delta^D_{in}$ is the  Friedrichs  selfadjoint extension of the Dirichlet Laplacian on $\mathcal M_{in}=\{p\in\mathcal M: |x|\leq R\}$.
We  introduce the Dirichlet-to-Neumann  operator
\begin{equation}\label{DtN}
\mathcal N(\mu^2) f=\lim_{x\to R+}\bigl\langle-\partial_x u (\mu;-x),\partial_x u (\mu;x) \bigr\rangle+\lim_{x\to R-}\bigl\langle\partial_x u (\mu;-x),- \partial_x u (\mu;x)\bigr\rangle;
\end{equation}
here $\langle\cdot,\cdot\rangle\in L^2(\Gamma_-)\oplus L^2(\Gamma_+)\equiv L^2(\Gamma)$ with $\Gamma_\pm=\{p\in\mathcal M: x=\pm R\}$.


\begin{thm}\label{N} Let $\epsilon>0$ be sufficiently small. Then the functions
$$
\mu\mapsto \mathcal N(\mu^2)\in \mathcal B (H^1(\Gamma); L^2(\Gamma)),\quad\mu\mapsto \mathcal N(\mu^2)^{-1}-\frac {i}{2\mu} \bigl(\cdot,1\bigr)_{L^2(\Gamma)}\in\mathcal B(L^2(\Gamma)),
$$ and $\mu\mapsto \det \mathcal N(\mu^2)\in \Bbb C$ are holomorphic  in the disc $|\mu|<\epsilon$, where $\det \mathcal N(\mu^2)$ is the zeta regularized determinant of $\mathcal N(\mu^2)$.  Moreover, as $\mu$ tends to zero  we have
\begin{equation}\label{detN}
\det \mathcal N(\mu^2)=-i\mu(\det\!^*\, \mathcal N(0) +O(\mu)),
\end{equation}
where $\mathcal N(0)$ has zero as an eigenvalue and
$\det\!^*\, \mathcal N(0)\in \Bbb R$ is the corresponding zeta regularized determinant
with zero eigenvalue excluded.
%

\end{thm}

\begin{proof} The main idea of the proof is essentially the same as in~\cite[Theorem B*]{BFK} and~\cite[Theorem B]{Lee}.

First we show that $\mathcal N(\mu^2)\in \mathcal B(H^1(\Gamma),L^2(\Gamma))$ is holomorphic in $\mu$, $|\mu|<\epsilon$. Since $\Delta^D_{in}$ is a positive selfadjoint operator, its resolvent $(\Delta^D_{in}-\mu^2 )^{-1}: H^{-1/2}(\mathcal M_{in})\to H^{3/2}(\mathcal M_{in})$ is a holomorphic function of $\mu^2$ in the sufficiently small disc $|\mu^2|<\epsilon^2$; here $\|v; H^{s}(\mathcal M_{in})\|=\|(\Delta^D_{in})^{s/2}v; L^2(\mathcal M_{in})\|$. Let  $\tilde f\in H^{3/2}(\mathcal M_{in})$ be a continuation of $f\in H^1(\Gamma)$.  Then in the small disc $|\mu|<\epsilon$ the equality~\eqref{u} defines a holomorphic family of operators mapping  $H^1(\Gamma)\ni f\mapsto u(\mu)\in H^{3/2}(\mathcal M_{in})$. As a consequence, for any  $f\in H^1(\Gamma)$ the second limit in~\eqref{DtN} is a holomorphic function of $\mu$ (more precisely of  $\mu^2$), $|\mu|<\epsilon$, with values in $L^2(\Gamma)$. The first limit in~\eqref{DtN} also defines a holomorphic with respect to $\mu$ operator in $\mathcal B(H^1(\Gamma),L^2(\Gamma))$ as it is seen from the explicit formulae
$$
\begin{aligned}
\partial_x u\bigl(\mu;&\pm(R+),y\bigr)=  \pm\Bigl(i\mu\int_{O_k} f(y')\,dy'
\\
& -\sum_{\ell\in \Bbb Z\setminus\{0\} } \sqrt{4\pi^2\ell^2|O_k|^{-2}-\mu^2} \int_{O_k} f(y') e^{ 2\pi i \ell |O_k|^{-1} (y-y')}\,dy'\Bigr),\quad y\in O_k,\ 1\leq k\leq n,
\end{aligned}
$$
obtained by separation of variables in the cylindrical ends; here $+$ signs (resp. $-$) are taken if $O_k$ corresponds to a right (resp. left) cylindrical end.

On the next step we make use of the representation
\begin{equation}\label{NR}
\mathcal
N(\mu^2)^{-1}= \bigl((\Delta_0-\mu^2)^{-1}(\cdot\otimes\delta_\Gamma)\bigr)\upharpoonright_{\Gamma},
\end{equation}
 where $\delta_\Gamma$ is the Dirac $\delta$-function along
$\Gamma$, the action of the resolvent on $(\cdot\otimes\delta_\Gamma)$ is understood in the sense of distributions,
 and $\upharpoonright_{\Gamma}$ is the restriction map to $\Gamma$; for a proof of~\eqref{NR}  see~\cite[Proof of Theorem2.1]{Carron}.

  Let $\varrho$ be a smooth cutoff function on $\mathcal M$ supported in a small neighborhood of $\Gamma$ and such that $\varrho=1$ in a vicinity of $\Gamma$. Since $\varrho$ is supported outside of conical points, the local elliptic coercive estimate
\begin{equation}\label{APR}
 \|\varrho u ; H^1(\mathcal M)\|\leq C(\|\tilde\varrho\Delta u; H^{-1}(\mathcal M)\|+\|\tilde \varrho u; L^2(\mathcal M)\|
 \end{equation}
 is valid, where $\tilde \varrho\in C^\infty_0(\mathcal M\setminus\mathcal P)$ and $\varrho\tilde\varrho=\varrho$. In particular, for
 $$
 u=\Phi(\mu)f:=\bigl(\Delta_0-\mu^2\bigr)^{-1}f-\frac{i}{2\mu}\bigl(f,1\bigr)_{L^2(\mathcal M)}
 $$
 \eqref{APR} implies
$$
\|\varrho \Phi(\mu)f ; H^1(\mathcal M)\|\leq C\bigl(\|f;L^2_\epsilon(\mathcal M)\|+|\mu|^2\|(\Delta_0-\mu^2)^{-1} f; L^2_{-\epsilon}(\mathcal M)\|+\|\Phi(\mu)f; L^2_{-\epsilon}(\mathcal M)\|\bigr).
 $$
 This together with Theorem~\ref{RSP} shows that  $\mu\mapsto \varrho \Phi(\mu)\in \mathcal B\bigl( L^2_{\epsilon}(\mathcal M); H^1(\mathcal M)\bigr)$ is holomorphic in the disc $|\mu|<\epsilon$. Since the mapping $L^2(\Gamma)\ni \psi\mapsto \psi\otimes\delta_\Gamma\in H^{-1}(\mathcal M)=(H^{1}(\mathcal M))^*$ is continuous,  for any $f\in L^2_\epsilon(\mathcal M)$ we have 
$$\begin{aligned}
\bigl( (\Delta_0-{\mu}^2)^{-1}(\cdot\otimes\delta_\Gamma), f\bigr)_{L^2(\mathcal M)}=\Bigl(\cdot\otimes\delta_\Gamma, \frac{-i}{2\bar\mu}\bigl(f,1\bigr)_{L^2(\mathcal M)}+\Phi(-\bar\mu)f\Bigr)_{L^2(\mathcal M)}\\ =\frac {i}{2\mu} \bigl(\cdot,1\bigr)_{L^2(\Gamma)}\bigl(1,f\bigr)_{L^2(\mathcal M)}+\bigl(\cdot\otimes\delta_\Gamma,\varrho \Phi(-\bar\mu)f\bigr)_{L^2(\mathcal M)},
\end{aligned}
$$
where $(\cdot,\cdot)_{L^2(\mathcal M)}$ is extended to the pairs in $H^{-1}(\mathcal M)\times H^{1}(\mathcal M)$ and $L^2_{-\epsilon}(\mathcal M)\times L^2_\epsilon(\mathcal M)$.
In other words, the equality
\begin{equation}\label{au1}
 (\Delta_0-{\mu}^2)^{-1}(\cdot\otimes\delta_\Gamma)=\frac {i}{2\mu} \bigl(\cdot,1\bigr)_{L^2(\Gamma)}+\mathfrak H(\mu),\ |\mu|<\epsilon,
\end{equation}
holds in $L^2_{-\epsilon}(\mathcal M)$, where  $\mu\mapsto\mathfrak H(\mu)\in \mathcal B\bigl(L^2(\Gamma),L^2_{-\epsilon}(\mathcal M)\bigr)$ is holomorphic. We substitute $u= \mathfrak H(\mu)\psi$ into~\eqref{APR}  and obtain
 $$
 \begin{aligned}
\|\varrho\mathfrak H(\mu)\psi; H^1(\mathcal M)\|\leq& C\bigl(\|\psi\otimes \delta_\Gamma; H^{-1}(\mathcal M)\|\\&+|\mu|^2\|(\Delta_0-\mu^2)^{-1}(\psi\otimes\delta_\Gamma); L^2_{-\epsilon}(\mathcal M)\|+\|\mathfrak H(\mu)\psi; L^2_{-\epsilon}(\mathcal M)\|\bigr).
\end{aligned}
$$
Thus  $\mu\mapsto \varrho\mathfrak H(\mu)\in \mathcal B(L^2(\Gamma); H^1(\mathcal M)\bigr)$ is holomorphic.
Now from~\eqref{au1},~\eqref{NR}, and continuity of the embedding $H^1(\mathcal M)\upharpoonright_{\Gamma}\hookrightarrow  L^2(\Gamma)$ we conclude that
\begin{equation}\label{DRepr}
\mathcal N(\mu^2)^{-1}=\frac {i}{2\mu} \bigl(\cdot,1\bigr)_{L^2(\Gamma)}+\mathfrak H_\Gamma(\mu),\quad |\mu|<\epsilon,
\end{equation}
where $\mathfrak H_\Gamma(\mu)\psi=(\varrho\mathfrak H(\mu)\psi)\upharpoonright_{\Gamma}$ and $\mu \mapsto \mathfrak H_\Gamma(\mu)\in \mathcal B(L^2(\Gamma))$ is holomorphic. In particular,~\eqref{DRepr} implies that zero is a simple eigenvalue of $\mathcal N(0)$ and $\ker \mathcal N(0)=\{c\in \Bbb C\}$; cf.~\eqref{SP}.

 The operator $\mathcal N(\mu^2)$ is an elliptic classical  pseudodifferential  operator on $\Gamma$ (all conical points of $\mathcal M$ are outside of $\Gamma$ and thus do not affect properties of the   symbol of $\mathcal N(\mu^2)$), e.g. from~\eqref{NR} one can see that the  principal symbol of $\mathcal N(\mu^2)$ is $2|\xi|$. Besides,~\eqref{NR} implies that  $\mathcal N(\mu^2)$ with $\mu^2\leq 0$ is formally selfadjoint, and $\mathcal N(\mu^2)$  is positive if $\mu^2<0$ and nonnegative if $\mu=0$. Therefore the closed unbounded operator $\mathcal N(\mu^2)$ in $L^2(\Gamma)$ with domain $H^1(\Gamma)$ is selfadjoint for $\mu^2\leq 0$, it is positive if $\mu^2<0$, and nonnegative if $\mu=0$, e.g.~\cite{Shubin}.

Let $\mu\in i[0,\epsilon)$ and let  $ 0\leq \lambda_1(\mu)\leq\lambda_2(\mu )\leq \lambda_3(\mu)\leq\cdots$ be the eigenvalues of the selfadjoint operator $\mathcal N(\mu^2)$. The operator $\mathcal N(\mu^2)$, and therefore its eigenvalues and eigenfunctions, are holomorphic functions of $\mu$ in the disc $|\mu|<\epsilon$ ; e.g.~\cite[Chapter VII]{Kato}).
Since  $\epsilon$ is sufficiently small,  the  eigenvalue $\lambda_1(\mu)$ remains simple for $|\mu|<\epsilon$ and $\lambda_1(\mu)\to 0$ as $\mu\to 0$, while all other eigenvalues satisfy $\delta<|\lambda_2(\mu)|\leq |\lambda_3(\mu)|\leq\cdots$ with some $\delta>0$.
Let $\psi(\mu)$ be the eigenfunction corresponding to the eigenvalue
$\lambda_1(\mu)$ of $\mathcal N(\mu^2)$ and satisfying $(\psi(\mu),\psi(-\bar\mu))_{L^2(\Gamma)}=1$; clearly,  $\psi(0)=2^{-1/2}$.   The equality
\begin{equation*}
1/\lambda_1(\mu)=\bigl(\mathcal
N(\mu^2)^{-1}\psi(\mu),\psi(-\bar \mu)\bigr)_{L^2(\Gamma)}, \quad \mu\in i (0,\epsilon)
\end{equation*}
extends by analyticity to the punctured disc $|\mu|<\epsilon$, $\mu\neq 0$. This together with~\eqref{DRepr} and $\|\psi(\mu)-  2^{-1/2};L^2(\Gamma)\|=O (|\mu|)$  gives
\begin{equation}\label{lambda}
\lambda_1(\mu)=-i\mu+O(\mu^2),\quad |\mu|<\epsilon.
\end{equation}

For $\mu\neq 0$ the operator $\mathcal N(\mu^2)$ is invertible, the function $\zeta (s)=\Tr \mathcal N(\mu^2)^{-s}$ is  holomorphic in  $\{s\in\Bbb C:\Re s>1\}$ and admits a meromorphic continuation to $\Bbb C$ with no pole at $s=0$. We set $\det \mathcal N(\mu^2)=e^{-\partial_s\zeta(0)}$. Besides, the function $\zeta^*(s)=\Tr \mathcal N_*(\mu^2)^{-s}$ of the invertible operator $$\mathcal N_*(\mu^2)= \mathcal N(\mu^2)+(1-\lambda_1(\mu))(\cdot, \psi(-\bar\mu))_{L^2(\Gamma)}\psi(\mu),\quad |\mu|<\epsilon,$$ is holomorphic in  $\{s\in\Bbb C:\Re s>1\}$ and has a meromorphic continuation to $\Bbb C$ with no pole at $s=0$; here the Riesz projection $(\cdot, \psi(-\bar\mu))_{L^2(\Gamma)}\psi(\mu)$ is a smoothing operator. We set $\det\!^*\, \mathcal N(\mu^2)=\det \mathcal N_*(\mu^2)=e^{-\partial_s\zeta^*(0)}$. Clearly,
\begin{equation}\label{D*}
\det \mathcal N(\mu^2)=\lambda_1(\mu)\det\!^*\,\mathcal N(\mu^2).
\end{equation}
Note that~\eqref{NR} also implies that the order of pseudodifferential operator $\partial^\ell_\mu \mathcal N(\mu^2)^{-1}$ is $-1-2\ell$. Thus the order of $\partial_\mu^\ell \mathcal N(\mu^2)$ is $1-2\ell$, for $\ell\geq 1$ and $|\mu|<\epsilon$ the operator $\partial^{\ell-1}_\mu\bigl[\bigl(\partial_\mu\mathcal N_*(\mu^2)\bigr)\mathcal N_*(\mu^2)^{-1}\bigr]$ is trace class  and
$$
\begin{aligned}
&\partial^\ell_\mu \log\det\!^*\, \mathcal N(\mu^2)=\Tr \bigl(\partial^{\ell-1}_\mu\bigl[\bigl(\partial_\mu\mathcal N_*(\mu^2)\bigr)\mathcal N_*(\mu^2)^{-1}\bigr]\bigr),
\\
&\partial_{\bar\mu} \log\det\!^*\, \mathcal N(\mu^2)=\Tr \bigl(\bigl(\partial_{\bar\mu}\mathcal N_*(\mu^2)\bigr)\mathcal N_*(\mu^2)^{-1}\bigr)=0;
\end{aligned}
$$
 see  \cite{Forman,BFK}. As a consequence,  $\mu\mapsto\det\!^*\, \mathcal N(\mu^2)$ is holomorphic in the disc $|\mu|<\epsilon$. This together with~\eqref{D*} and~\eqref{lambda} completes the proof.
\end{proof}

\subsection{Relative zeta function}
 Both perturbed and unperturbed Mandelstam diagrams can be considered as strips $\Pi$ and $\mathring\Pi$ with different slits. Therefore $L^2(\Pi)=L^2(\mathring\Pi)$ and the spaces $L^2(\mathcal M)$ and $L^2(\mathring{\mathcal M})$ can be naturally identified.  Starting from now on we consider only selfadjoint  Friedrichs  extensions $\Delta$ and $\mathring\Delta$  in $L^2(\mathcal M)$; in other words, we set $\epsilon=0$ and omit it from notations.
\begin{lem}\label{TR}
For all $t>0$ the operator $e^{-t\Delta}-e^{-t\mathring\Delta}$ is trace class and
\begin{equation}\label{m1}
\Tr(e^{-t\Delta}-e^{-t\mathring\Delta})=O(t^{-1/2})\text { as
} t\to+\infty.
\end{equation}

\end{lem}
\begin{proof} As is known~\cite[Theorem 2.2]{Carron},
$
(\Delta+1)^{-1}-(\Delta^D_{in}\oplus\Delta^D_{out}+1)^{-1}
$ is trace class; here $\Delta^D_{in}$ is the same as in the section~\ref{DTN} and $\Delta^D_{out}$ is the selfadjoint Friedrichs extension of the Dirichlet Laplacian  on
 $\mathcal M_{out}=\{p\in\mathcal M;|x|\geq R\}$, i.e.  $\Delta^D_{in}\oplus\Delta^D_{out}$ is the operator of the Dirichlet problem~\eqref{DP}.
Then by the Krein theorem, see e.g.~\cite[Chapter
8.9]{yafaev} or~\cite[Theorem 3.3]{Carron}, there exists a spectral
shift function $\xi\in L^1(\Bbb R_+, (1+\lambda)^{-2}\,d\lambda)$
such that
$$
\Tr\bigl( (\Delta+1)^{-1}-(\Delta^D_{in}\oplus\Delta^D_{out}+1)^{-1}
\bigr)=-\int_0^\infty \xi(\lambda)(1+\lambda)^{-2}\,d\lambda.
$$
Moreover, the following representation is valid
\begin{equation}\label{Krein}
\Tr\bigl( e^{-t\Delta}-e^{-t\Delta^D_{in}\oplus\Delta^D_{out}}
\bigr)=-t\int_0^\infty e^{-t\lambda} \xi(\lambda)\,d\lambda,
\end{equation}
where the right hand side is finite. Thus $ e^{-t\Delta}-e^{-t\Delta^D_{in}\oplus\Delta^D_{out}}$ is trace class. Besides,
 by~\cite[Theorem~3.5]{Carron} we have
 \begin{equation}\label{RELC}
 \xi(\lambda)=\pi^{-1} \arg\det \mathcal N (\lambda+i0),\quad  \lambda>0.
 \end{equation}
 This together with~\eqref{detN} gives $\xi(\lambda)=3/2+O(\sqrt\lambda)$ as $\lambda\to 0+$.  As a consequence, the right hand side of~\eqref{Krein} provides the left hand side with asymptotic
\begin{equation}\label{f1}
 \Tr\bigl( e^{-t\Delta}-e^{-t\Delta^D_{in}\oplus\Delta^D_{out}}
\bigr)=3/2+ O(t^{-1/2}) \text { as
} t\to+\infty.
 \end{equation}
Similarly we conclude that $e^{-t\mathring\Delta}-e^{-t\mathring\Delta^D_{in}\oplus\mathring\Delta^D_{out}}$ is trace class and
\begin{equation}\label{f2}
 \Tr\bigl( e^{-t\mathring\Delta}-e^{-t\mathring\Delta^D_{in}\oplus\mathring\Delta^D_{out}}
\bigr)=3/2+ O(t^{-1/2}) \text { as
} t\to+\infty;
\end{equation}
here the operators $\mathring\Delta^D_{in}$ and $\mathring\Delta^D_{out}$ of the Dirichlet problems on $\{p\in\mathring{\mathcal M}:|x|\leq R \}$ and $\{p\in\mathring{\mathcal M}:|x|\geq R \}$ respectively are introduced in the same way as  $\Delta^D_{in}$ and $\Delta^D_{out}$. For the operator $\Delta^D_{in}$ (resp. $\mathring\Delta^D_{in}$) on compact manifold it is known that $e^{-t\Delta^D_{in}}$ (resp. $e^{-t\mathring\Delta^D_{in}}$) is trace class and $\Tr e^{-t\Delta^D_{in}}= O(e^{-\lambda t})$ (resp. $\Tr e^{-t\mathring\Delta^D_{in}}= O(e^{-\lambda t})$) as $t\to+\infty$, where $\lambda>0$ is the first eigenvalue of $\Delta^D_{in}$ (resp. $\mathring\Delta^D_{in}$). Since $\Delta^D_{out}\equiv\mathring\Delta^D_{out}$, 
this together with~\eqref{f1} and~\eqref{f2}
completes the proof.
\end{proof}

\begin{Remark} In the general framework~\cite{Mueller} (see also~\cite{MulMul,Mu4}) the long time behavior of $\Tr\bigl( e^{-t\Delta}-e^{-t\Delta^D_{in}\oplus\Delta^D_{out}}
\bigr)$ is supposed to be studied via properties of the corresponding scattering matrix near the bottom of the continuous spectrum (as it naturally follows from~\eqref{Krein} and the Birman-Krein theorem). In contrast to this, in the proof of Lemma~\ref{TR} we follow the original idea of Carron~\cite[Theorem~3.5]{Carron},~\cite{Carron+}  and immediately obtain the result relying on~\eqref{RELC} and~\eqref{detN}.
\end{Remark}

\begin{lem}\label{tr zero} Let $K$ be the number of the interior slits  of the diagram ${\cal M}$. Then for some $\delta>0$
\begin{equation}\label{res}
 \Tr (e^{-t\Delta}-e^{-t\mathring\Delta})=-K/4+O(e^{-\delta/t})
\end{equation}
as $t\to 0+$.
\end{lem}
\begin{proof} Let $\{{\cal U}_j\}$ be a finite covering of the flat surface ${\cal M}$ by open discs centered at conical points, flat open discs,
and open semi-infinite cylinders and let $\{\zeta_j\}$ be the $C^\infty$ partition of unity subject to this covering. Let also $\tilde \zeta_j$
be smooth functions supported in small neighborhoods of ${\cal U}_j$ such that $\zeta_j \tilde \zeta_j=\zeta_j$ and
$${\rm dist} ({\rm supp}\nabla \tilde \zeta_j, {\rm supp} \zeta_j)>0$$
for all $j$. Define a parametrix for the heat equation on ${\cal M}$ as

\begin{equation}\label{par ker}
\begin{aligned}
\mathcal P(p, q; t)=\sum_{j}\tilde \zeta_j(p)\mathcal
K_j(p, q; t)\zeta_j(q),
\end{aligned}
\end{equation}
where $\mathcal K_j$ is (depending on the type of the element ${\cal U}_j$ of the covering) either  the heat kernel on the infinite flat
cone with conical angle $4\pi$ (see, e. g.,  \cite{KokKor}, f-la (4.4)) or the standard heat kernel in ${\mathbb R}^2$ or
the heat kernel
\begin{equation}\label{par1}
H(x,y, x', y', t)=\frac{e^{-(x-x')^2/(4t)}}{\sqrt{4\pi
t}a}
 \sum_{n\in{\mathbb Z}} e^{i 2\pi
a^{-1}n(y-y')-4\pi^2n^2a^{-2} t}
\end{equation}
in the infinite cylinder with circumference $a$ (the latter is the same as the circumference of the corresponding
semi-infinite cylinder from the covering). One has the relation
$$
\lim_{t\downarrow 0} \int_{\mathcal M} \mathcal P(x,y,x',y';t)
f(x',y')\,d x'\,dy'= f(x, y),\quad \forall f\in C_0^\infty(\mathcal
M)
$$
and the estimate
 \begin{equation}\label{est}
 |\mathcal P_1(x,y,x',y';t)|\leq C e^{-\delta(1+x^2+x'^2)/t}
 \end{equation}
for ${\mathcal P}_1(p, q; t):=(\partial_t-\Delta){\mathcal P}(p, q; t)$ and some $\delta>0$. To prove (\ref{est}) one has to notice that
${\mathcal P}_1(p, q; t)$ vanishes when $p$ does not belong to the union of ${\rm supp} \nabla \zeta_j$ (which is a compact subset of ${\cal M}$) or
when the distance between $p$ and $q$ is sufficiently small and then make use of the explicit expressions for the standard heat kernels in (\ref{par ker}).
Due to (\ref{est}) one can construct the heat kernel on ${\cal M}$ in the same way as it is usually done for compact manifolds (see, e. g. \cite{MinPle}).
 We
introduce consecutively
\begin{equation}\label{cons}
\mathcal P_{\ell+1}(x,y,x',y';t)=\int_0^t\int_{\mathcal M}\mathcal
P_1(x,y,\hat x,\hat y;t-\hat t)\mathcal P_\ell (\hat x,\hat y, x',y';
\hat t)\,d\hat x\,d\hat y \, d\hat t,\quad \ell\geq 1.
\end{equation}
By \eqref{est} the second integral in~\eqref{cons} is
absolutely convergent and
\begin{equation}\label{est1}
|\mathcal P_{\ell+1}(x,y,x',y';t)|\leq e^{-\delta(1+x^2+x'^2)/t}(c
t)^\ell
\end{equation}
for some  $c>0$. For small $t$ the heat kernel $\mathcal H$ on ${\cal M}$ is given by
\begin{equation}\label{h}
\mathcal H=\mathcal P+\sum_{\ell=1}^\infty (-1)^\ell\mathcal
P_\ell.
\end{equation}

Moreover, one has the following estimate for the difference between the heat kernel and the parametrix ${\mathcal P}$
\begin{equation}\label{HP}
|\mathcal H(x,y,x',y';t)-\mathcal P(x,y,x',y';t)|\leq C
e^{-\delta(1+x^2+x'^2)/t},
\end{equation}
where  $t>0$ is sufficiently small, $\delta$ and $C$ are some positive constants.

Similarly one can construct a parametrix $\mathcal Q$ and the
heat kernel $\mathring{\mathcal H}$ for  the "free" diagram $\mathring{\cal M}$ (coinciding with ${\cal M}$ for $|x|>R$, with sufficiently large $R$).
Obviously, $\mathcal P(p, p; t)=\mathcal Q (p, p; t)$ for $p=(x, y)$, $|x|>R$. Thus,
$$
\begin{aligned}
 \Tr (e^{-t\Delta}-e^{-t\mathring\Delta})=\int_{\mathcal M} \mathcal H(x,y,x,y;t)dx\,dy-\int_{\mathring{\mathcal M}}\mathring{\mathcal
 H}(x,y,x,y;t)\,dx\,dy\\
 =\int_{{\cal M}\cap\{|x|<R\}} \mathcal H(x,y,x,y;t)dx\,dy-\int_{\mathring{\cal M}\cap\{|x|<R\}}\mathring{\mathcal
 H}(x,y,x,y;t)\,dx\,dy+O(e^{-\delta/t}),
\end{aligned}
$$
where $\delta>0$ and $t\downarrow 0$.
From (\cite{KokKor}, Theorem 8) it follows that
the first integral at the right
has the asymptotics
$$\frac{{\rm Area}({\cal M}\cap\{|x|<R\}}{4\pi t}+ \frac{1}{12}\sum_k(\frac{2\pi}{\beta_k}-\frac{\beta_k}{2\pi})+O(e^{-\delta/t}),$$
as $t\to 0+$, where the summation is over the conical points of ${\cal M}$ inside   $\{|x|<R\}$ and all the conical angles $\beta_k$ are equal to $4\pi$.
The second term has the similar asymptotics with ${\rm Area}(\mathring{\cal M}\cap\{|x|<R\}={\rm Area}({\cal M}\cap\{|x|<R\}$ and
smaller number of conical points (by $2K$, where $K$ is the number of interior slits of the perturbed diagram ${\cal M}$). This implies (\ref{res}).
\end{proof}

Now we are in position to introduce the relative zeta determinant $\det (\Delta-\mu^2,\mathring\Delta-\mu^2)$ following~\cite{Mueller}.
As a consequence of~Lemma~\ref{TR}
the  function
$$
\zeta_\infty(s; \Delta-\mu^2,\mathring\Delta-\mu^2)=\frac{1}{\Gamma(s)}\int_1^\infty
t^{s-1}e^{t\mu^2}\Tr
(e^{-t\Delta}-e^{-t\mathring\Delta})\, dt,\quad \mu^2\leq 0,
$$
is holomorphic in  $\{s\in\Bbb C: \Re s<1/2\}$ (and $\zeta_\infty(0; \Delta,\mathring\Delta)=0$).
Lemma~\ref{tr zero} implies that the holomorphic in $\{s\in\Bbb C: \Re s>1\}$ function
$$\zeta_0(s; \Delta-\mu^2,\mathring\Delta-\mu^2)=\frac{1}{\Gamma(s)}\int_0^1  t^{s-1}e^{t\mu^2}\Tr
(e^{-t\Delta}-e^{-t\mathring\Delta})\, dt,\quad \mu^2\leq 0,
$$
has a meromorphic extension to $s\in\Bbb C$ with no pole at $s=0$  (and $\zeta_0(0; \Delta,\mathring\Delta)=-K/4$). We introduce the relative zeta function $$\zeta(s; \Delta-\mu^2,\mathring\Delta-\mu^2)=\zeta_0(s; \Delta-\mu^2,\mathring\Delta-\mu^2)+\zeta_\infty(s; \Delta-\mu^2,\mathring\Delta-\mu^2)$$ and the corresponding relative determinant
$$
\det (\Delta-\mu^2,\mathring\Delta-\mu^2)=e^{- \partial_s\zeta
(0; \Delta-\mu^2,\mathring\Delta-\mu^2)},\quad\mu^2\leq 0.
$$

\begin{thm}
\begin{equation}\label{det}
\det(\Delta-\mu^2,\mathring\Delta-\mu^2)=\det(\Delta,\mathring\Delta)+o(1)\text{
as }\mu^2\to 0-.
\end{equation}
\end{thm}
\begin{proof}  From  analytic continuations of $\zeta_0$ and $\zeta_\infty$  it is easily
seen that as $\mu^2\to 0-$ we have
$$
\partial_s\zeta_0(0;\Delta-\mu^2,\mathring\Delta-\mu^2)=
\partial_s\zeta_0(0;\Delta,\mathring\Delta)+o(1),
$$
$$
\partial_s\zeta_\infty(0;\Delta-\mu^2,\mathring\Delta-\mu^2)=
\partial_s\zeta_\infty(0;\Delta,\mathring\Delta)+o(1),
$$
 which proves the assertion.
\end{proof}

\subsection{Decomposition formula}
\begin{proof}[Proof of Theorem~\ref{main}] The asymptotic $\Tr
(e^{-t\Delta}-e^{-t\Delta^D_{in}\oplus\Delta^D_{out}})\sim \sum_{j\geq- 2} a_j  t^{j/2}$ as $t\to 0+$ (which can be established in the same way as~\eqref{res}) together with~\eqref{f1}
implies that the relative zeta function
$$
\zeta\bigl(s;\Delta-\mu^2,
\Delta^D_{in}\oplus\Delta^D_{out}-\mu^2\bigr)=\frac{1}{\Gamma(s)}\int_0^\infty t^{s-1}e^{t\mu^2}\Tr
(e^{-t\Delta}-e^{-t\Delta^D_{in}\oplus\Delta^D_{out}})\, dt, \quad\mu^2<0,
$$
is holomorphic for $\{s\in\Bbb C; \Re s>1\}$ and has a meromorphic extension to $s\in\Bbb C$ with no pole at $s=0$. We set $$\det \bigl(\Delta-\mu^2,
\Delta^D_{in}\oplus\Delta^D_{out}-\mu^2\bigr)=e^{-\partial_s\zeta\bigl(0;\Delta-\mu^2,
\Delta^D_{in}\oplus\Delta^D_{out}-\mu^2\bigr)}.$$
Similarly we define $\det \bigl(\mathring\Delta-\mu^2,
\mathring\Delta^D_{in}\oplus\Delta^D_{out}-\mu^2\bigr)$.
Then by  \cite[Theorem 4.2]{Carron} we have
$$
\det \bigl(\Delta-\mu^2,
\Delta^D_{in}\oplus\Delta^D_{out}-\mu^2\bigr)=\det\mathcal
N(\mu^2);\ \
\det \bigl(\mathring\Delta-\mu^2,
\mathring\Delta^D_{in}\oplus\Delta^D_{out}-\mu^2\bigr)=\det\mathring{\mathcal
N}(\mu^2).
$$
 Dividing the first equality by the second one we get
\begin{equation}\label{*}
\frac{\det(\Delta-\mu^2,\mathring\Delta-\mu^2)\det(\mathring\Delta^D_{in}-\mu^2)}{\det(\Delta^D_{in}-\mu^2)}=\frac{\det\mathcal
N(\mu^2)}{\det\mathring{\mathcal N}(\mu^2)},
\end{equation}
where $\det(\Delta^D_{in}-\mu^2)$ and $\det(\mathring\Delta^D_{in}-\mu^2)$ are the zeta regularized determinants of Dirichlet Laplacians on compact manifolds. Since $\Delta^D_{in}$ is positive, we have $\det(\Delta^D_{in}-\mu^2)\to \det\Delta^D_{in}$ as $\mu^2\to 0$, and the same is true for $\mathring\Delta^D_{in}$.
Thanks to~\eqref{det} and~Theorem~\ref{N} applied to $\mathcal N(\mu^2)$
and $\mathring{\mathcal N}(\mu^2)$ we can pass in~\eqref{*} to the limit as
$\mu^2\to0-$ and obtain
$$
\frac{\det(\Delta,\mathring\Delta)\det\mathring\Delta^D_{in}}{\det\Delta^D_{in}}=\frac{\det\!^*\,
\mathcal N(0)}{\det\!^*\, \mathring{\mathcal N}(0)}.
$$
 Since
$\det\mathring\Delta^D_{in}$ and $\det\!^*\, \mathring{\mathcal N}(0)$ are moduli
 independent, this proves Theorem~\ref{main}, where $C=(\det\mathring\Delta^D_{in}\det\!^*\, \mathring{\mathcal
 N}(0))^{-1}$ and $\mathcal N=\mathcal N(0)$.
\end{proof}

\section{Variational formulas for the relative determinant}
\subsection{Compactification of the ends}
In the holomorphic local parameter $\zeta_k=\exp(\mp 2\pi z/|O_k|)$, $z=x+iy$ in a vicinity $U_k=\{x>R\}$ (or $\{x<-R\}$) of the point at
infinity $\zeta_k=0$ of the $k$-th cylindrical end of ${\cal M}$ the flat metric ${\bf m}$ on ${\cal M}$ is written in the form
$${\bf m}=\frac{|O_k|^2}{4\pi^2}\frac{|d\zeta_k|^2}{|\zeta_k|^2}\,.$$
Let $\chi_k$ be a smooth function on ${\mathbb C}$ such that $\chi_k(\zeta)=\chi_k(|\zeta|)$, $|\chi_k(\zeta)|\leq 1$,  $\chi_k(\zeta)=0$ if
$|\zeta|>\exp(-2\pi (R+1)/|O_k|)$ and $\chi(\zeta)=1$
if $|\zeta|<\exp(-2\pi (R+2)/|O_k|)$.
Introduce another metric $\tilde {\bf m}$ on ${\cal M}$ by
$$\tilde{\bf m}=\begin{cases} {\bf m}\ \text{ for}\ \ |x|<R;\\
[1+(|\zeta_k|^2-1)\chi(\zeta_k)]{\bf m}\ \text {in} \ \ U_k.   \end{cases}$$

Applying BFK decomposition formula~\cite[ Theorem B*]{BFK} to the determinant of the Laplacian  $\Delta^{\tilde{\bf m}}$  on the compact Riemannian manifold
$({\cal M}, \tilde {\bf m})$,
 we get
\begin{equation}\label{MAINtwo}
\log \det \Delta^{\tilde{\bf m}}=\log C_0+\log \det \Delta_{in}^D +\log \det\!^*\mathcal N+ \log\det \Delta^{\tilde{\bf m}}_{ext}, \end{equation}
 where $\Delta^{\tilde{\bf m}}_{ext}$ is the operator of the Dirichlet problem for $\Delta^{\tilde{\bf m}}$ in ${\cal M}\setminus \{|x|<R\}$,
 $$C_0=\frac{{\rm Area}({\cal M}, \tilde {\bf m})}{\sum |O_k|}\,,$$
and  $\mathcal N$ is the same as in (\ref{MAINone}).
From \eqref{MAINone} and \eqref{MAINtwo} it follows that $\log \det \Delta^{\tilde{\bf m}}$ and $\log \det (\Delta, \mathring\Delta)$ have the same variations
with respect to moduli $h_k$,
$\theta_k$, $\tau_k$ and, therefore,
$$\det \Delta^{\tilde{\bf m}}=C \det (\Delta, \mathring\Delta)$$
with moduli independent factor $C$.
\subsection{Variational formulas for resolvent kernel}
Denote by $G(\cdot, \cdot ; \lambda)$ the resolvent kernel of the Laplace operator $\Delta^{\tilde{\bf m}}$. From now on we assume that the spectral parameter $\lambda$ is real, so  $G(\cdot , \cdot ; \lambda)$ is a
real-valued function.

Introduce the one-form $\omega$ on ${\cal M}$
\begin{equation}
\omega=G(P, z, \bar z; \lambda)G_{z\bar z}(Q, z, \bar z; \lambda)d\bar z+G_z(P, z, \bar z; \lambda)G_z(Q, z, \bar z; \lambda)dz\,.
\end{equation}
Clearly,  $d\omega=0$ on ${\cal M}\cap \{|x|<R\}$.

The following proposition describes the variations of the resolvent kernel $G(P, Q; \lambda)$ under variations of moduli parameters. It is assumed that
positions of the points $P$ and $Q$ on the diagram are kept fixed when the moduli vary.

\begin{prop}\label{ResKer}
\begin{equation}\label{twist}
\frac{\partial G(P, Q; \lambda)}{\partial \theta_k}=4\Re\big\{\oint_{\gamma_k}\omega \big\}, \quad k=1, \dots, 3g+n-3\,;
\end{equation}
\begin{equation}\label{shift}
\frac{\partial G(P, Q; \lambda)}{\partial h_k}=-4\Re\big\{\oint_{b_k}\omega \big\},\quad k=1, \dots, g\,;
\end{equation}
\begin{equation}\label{stretch}
\frac{\partial G(P, Q; \lambda)}{\partial \tau_k}=4\Im\big\{\oint_{\pm A_k\pm A'_k\mp C_k}\omega \big\}\,.
\end{equation}
Here $\gamma_k$ are the contours along which the twists $\theta_k$ are performed, $b_k$ are $b$-cycles on the Riemann surface ${\cal M}$
 encircling the finite cuts of the diagram, contours $A_k$, $A'_k$ and $C_k$ coincide with circumferences of the three cylinders joining at the moment of "time"  $x=\tau_k$, the choice of sign $\pm$ depends on the position of the cylinders (two at the left and one at the right or vice versa), see Fig. 3.

\end{prop}

\begin{figure}
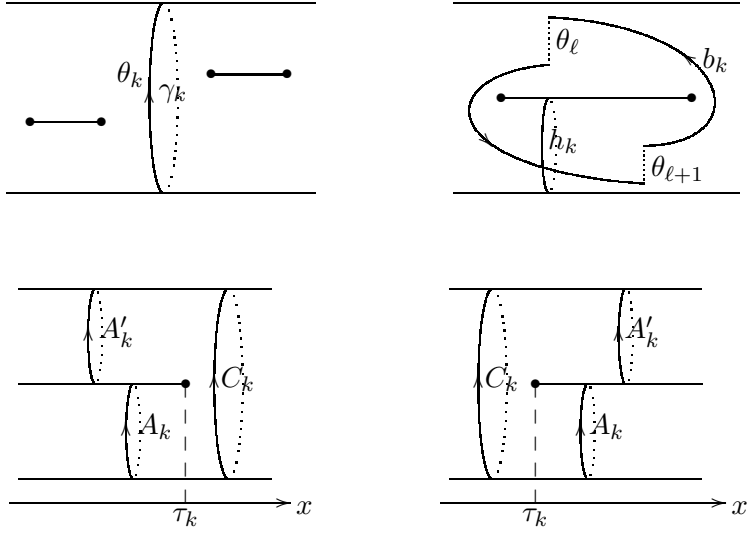
\centering
\[\xy0;/r.15pc/: (90,-15)*{\xy
  {\ar@{-}(17,-20)*{}; (70,-20)*{}}; {\ar@{-}(17,20)*{}; (70,20)*{}};
{\ar@{->} (15,-25)*{}; (75,-25)*{}}; (78,-26)*{x};
(13,0)*\ellipse(3,20){.};
(13,0)*\ellipse(3,20)__,=:a(-180){-};{\ar@{->}(23,0);(23,2)};(28,1)*{C_k};
{\ar@{-}(35,0);(70,0)};(35,0)*{ {\scriptstyle\bullet}};
(27,5)*\ellipse(1.5,10){.};
(27,5)*\ellipse(1.5,10)__,=:a(-180){-};
(23,-5)*\ellipse(1.5,10){.};
(23,-5)*\ellipse(1.5,10)__,=:a(-180){-};
{\ar@{--}(35,0);(35,-25)};(35,-28)*{\tau_k};
{\ar@{->}(44.5,-10);(44.5,-8)};(50,-9)*{A_k};{\ar@{->}(52.5,10);(52.5,12)};(57.5,11)*{A_k'};
\endxy};
(0,-15)*{\xy
  {\ar@{-}(-17,-20)*{}; (-70,-20)*{}}; {\ar@{-}(-17,20)*{}; (-70,20)*{}};
{\ar@{->} (-72,-25)*{};(-13,-25)*{}}; (-10,-26)*{x};
(-27,5)*\ellipse(1.5,10)__,=:a(180){-};
(-27,5)*\ellipse(1.5,10){.};
(-23,-5)*\ellipse(1.5,10){.};
(-23,-5)*\ellipse(1.5,10)__,=:a(-180){-};
(-13,0)*\ellipse(3,20){.};
(-13,0)*\ellipse(3,20)__,=:a(-180){-};{\ar@{->}(-29,0);(-29,2)};(-24,1)*{C_k};
{\ar@{-}(-35,0);(-70,0)};(-35,0)*{ {\scriptstyle\bullet}};
{\ar@{--}(-35,0);(-35,-25)};(-35,-28)*{\tau_k};
{\ar@{->}(-47.5,-10);(-47.5,-8)};(-41.5,-9)*{A_k};{\ar@{->}(-55.5,10);(-55.5,12)};(-49.5,11)*{A_k'};
\endxy};
(90,50)*{\xy
  {\ar@{-}(20,-20)*{}; (80,-20)*{}};
  {\ar@{-}(20,20)*{}; (80,20)*{}}; (75,8)*{b_k}; {\ar@{->}(70,8);(68,9.35)}; {\ar@{->}(27,-9);(28,-10)};
{\ar@{-}(30,0);(70,0)};(30,0)*{ {\scriptstyle\bullet}};(70,0)*{ {\scriptstyle\bullet}};
(20,-5)*\ellipse(1.5,10){.};
(20,-5)*\ellipse(1.5,10)__,=:a(-180){-};(43,-9)*{h_k};
(40,7);(60,-18)**\crv{(15,5)&(15,-15)}; {\ar@{.}(40,17);(40,7)}; (44,12)*{\theta_\ell};
(60,-10);(40,17)**\crv{(84,-10)& (80,15)}; {\ar@{.} (60,-18);(60,-10)}; (67,-15)*{\theta_{\ell+1}};
\endxy};
(0,50)*{\xy
  {\ar@{-}(-45,-20)*{}; (20,-20)*{}}; {\ar@{-}(-45,20)*{}; (20,20)*{}};
(-6,0)*\ellipse(3,20){.};
(-6,0)*\ellipse(3,20)__,=:a(-180){-};{\ar@{->}(-15,0);(-15,2)};(-19,5)*{\theta_k};(-10,1)*{\gamma_k};
{\ar@{-}(-2,5);(14,5)};(-2,5)*{ {\scriptstyle\bullet}}; (14,5)*{{\scriptstyle\bullet}};
{\ar@{-}(-40,-5);(-25,-5)};(-25,-5)*{{\scriptstyle\bullet}};(-40,-5)*{{\scriptstyle\bullet}};
\endxy};
\endxy\]
\caption{Contours.}\label{fig4}
\end{figure}

In the proof of Proposition~\ref{ResKer} we will make use of the representation of a solution to the homogeneous Helmholtz equation given by the following Lemma.
\begin{lem}\label{Kok-1} Let $G(z, \bar z, \xi, \bar \xi; \lambda)$ be the resolvent kernel of the operator $\Delta^{\tilde{\bf m}}$,
and let $u$
be a solution to
\begin{equation}\label{Helm}\Delta^{\tilde{\bf m}}u-\lambda u=0\end{equation} in ${\mathcal M}$. Let also $\Omega\subset {\cal M}$ be a an open subset of ${\cal M}$ with
piece-wise smooth boundary. Then for any $P\in \Omega$, $P=(\xi, \bar \xi)$ one has the relation
\begin{equation}\label{Green}
u(\xi, \bar \xi)=-2i\int_{\partial \Omega}G(z, \bar z, \xi, \bar \xi; \lambda)u_{\bar z}(z, \bar z)d\bar z+G_z(z, \bar z, \xi, \bar \xi;
\lambda)u(z, \bar z)dz\,.
\end{equation}
\end{lem}

\begin{proof}[Proof of Lemma~\ref{Kok-1}] Applying Stokes theorem to the integral over the boundary of the domain $\Omega_\epsilon=\Omega\setminus \{|z-\xi|\leq \epsilon\}$,
one gets the relation
$$\int_{\partial \Omega_\epsilon}G(z, \bar z, \xi, \bar \xi; \lambda)u_{\bar z}(z, \bar z)d\bar z+G_z(z, \bar z, \xi, \bar \xi;
\lambda)u(z, \bar z)dz=\iint_{\Omega_{\epsilon}}(Gu_{z\bar z}-G_{z\bar z}u)\,dz\wedge d\bar z$$
$$=\iint_{\Omega_{\epsilon}}\frac{1}{\rho(z, \bar z)}\left\{G(\Delta^{\tilde{\bf m}}u-\lambda u)-(\Delta^{\tilde{\bf m}}G-\lambda G)u\right\}dz
\wedge d\bar z=0\,,$$
where $\Delta^{\tilde{\bf m}}=\rho(z, \bar z)\partial_z\partial_{\bar z}$ in the conformal local parameter $z$.
Sending $\epsilon$ to $0$ and using  the asymptotics
$$G(z, \bar z, \xi, \bar \xi; \lambda)=\frac{1}{2\pi}\log|z-\xi|+O(1)$$
as $z\to \xi$, one gets (\ref{Green}).
\end{proof}

\begin{proof}[Proof of Proposition~\ref{ResKer}]
Let us prove (\ref{twist}).

Let $\Omega$ be the surface ${\cal M}$ cut along the twist contour $\gamma_k$. Denote the differentiation with respect to $\theta_k$ by dot.
The function $\dot{G}(P, \cdot; \lambda)$ satisfies homogeneous Helmholtz equation (\ref{Helm}). (Note that the singularity
of $G(P, \cdot; \lambda)$ at $P$ disappears after differentiation with respect to $\theta_k$.)
Differentiating the relation
$$
G_-(P, z, \bar z; \lambda; \{\dots, \theta_k, \dots\})=G_+(P, z+i\theta_k, \overline{z+i\theta_k}, \lambda; \{\dots, \theta_k, \dots\})
$$
for the left and right limit values of $G(P, \cdot; \lambda)$ at the contour $\gamma_k$ (we remind the reader that $G$ implicitly depends on moduli, this dependence is
indicated in the previous formula), we get
\begin{equation}\label{diff1}
\dot{G}_-(P, z, \bar z; \lambda) =\dot{G}_+(P, z, \bar z; \lambda)+i(G_z(P, z, \bar z; \lambda)-G_{\bar z}(P, z, \bar z, \lambda)),\end{equation}
\begin{equation}\label{diff2}
(\dot{G}_{\bar z})_-(P, z, \bar z; \lambda) =(\dot{G}_{\bar z})_+(P, z, \bar z; \lambda)+
i(G_{z\bar z}(P, z, \bar z; \lambda)-G_{\bar z \bar z}(P, z, \bar z, \lambda)).
\end{equation}

Assuming that the contour $\gamma_k$ is not homologous to zero and
using  (\ref{Green}), (\ref{diff1}) and (\ref{diff2}), we get
$$\dot{G}(P, Q; \lambda)=2\int_{\gamma_k}G(z, \bar z, Q; \lambda)\left[G_{z\bar z}(P, z, \bar z; \lambda)-G_{\bar z \bar z}(P, z, \bar z, \lambda)\right]
d\bar z$$$$+G_z(z, \bar z, Q; \lambda)\left[ G_z(P, z, \bar z; \lambda)-G_{\bar z}(P, z, \bar z, \lambda)\right]dz$$
$$=2\int_{\gamma_k}\omega(P, Q)+\overline{\omega(P,Q)}-d\left(G(Q, z, \bar z; \lambda)G_{\bar z}(P, z, \bar z; \lambda)\right)=4\Re\left\{\int_{\gamma_k}\omega\right\}\,.$$

In the case of homologically trivial contour $\gamma_k$ dividing the diagram ${\cal M}$ into two parts, ${\cal M}_-$ and ${\cal M}_+$, one has, say,
for $Q\in {\cal M}_-$:

$$\dot{G}(P, Q; \lambda)=-2i\oint_{\gamma_k} G(z, \bar z, Q; \lambda)[\dot{G}_-]_{\bar z}(P, z, \bar z; \lambda)d\bar z+G_z(z, \bar z; Q; \lambda)\dot{G}_-
(P, z, \bar z; \lambda)dz,$$
$$-2i\oint_{\gamma_k} G(z, \bar z, Q; \lambda)[\dot{G}_+]_{\bar z}(P, z, \bar z; \lambda)d\bar z+G_z(z, \bar z; Q; \lambda)\dot{G}_+
(P, z, \bar z; \lambda)dz=0\,.$$
These two relations together with \eqref{diff1}, \eqref{diff2} imply \eqref{twist}.

To prove \eqref{stretch} we notice (leaving the detailed proof to the reader) that the infinitesimal horizontal shift
of a zero $P_k$ of the differential $\omega$ (or, equivalently, the variation of the interaction time $\tau_k$) is the same as the insertion (removal) of
the infinitesimal horizontal cylinders along the circumferences  $A_k$, $A'_k$ and $C_k$ of the three cylinders of the diagram ${\cal M}$ meeting at $P_k$.
It is easy to show that the variation of the resolvent kernel under each such insertion (removal) is given by $$4\Im\big\{\oint_{\gamma}\omega \big\}\,$$
where the $\gamma$ is the cycle of the insertion (removal). The orientation of the cycle $\gamma$ depends on its position with respect to the point $P_k$
(from the left or from the right).

 It should also be  noticed that the sum $\pm A_k\pm A_k'\mp C_k$ is homologous to a small circular contour surrounding the zero $P_k$ and (\ref{stretch}) could also be
  proved using real analyticity of the resolvent kernel with respect to the local parameter $\sqrt{z-z(P_k)}$ (cf. the proof of formula (4.24)
  in \cite{KokKor}).

The proof of (\ref{shift}) is similar to the proof of the formula (4.20) in \cite{KokKor}. We leave it to the reader.
\end{proof}

\subsection{Variational formulas for regularized determinant}
Choose a canonical basis of $a$ and $b$-periods on the compact Riemann surface ${\cal M}$, the corresponding basis of normalized holomorphic
differentials $\{v_k\}$; $\oint_{a_j}v_k=\delta_{jk}$ and introduce the corresponding matrix of $b$-periods
$${\mathbb B}=||\oint_{b_j}v_k||_{j, k=1, \dots, g}\,,$$
the prime form $E(P, Q)$, the canonical meromorphic bidifferential
$$W(P, Q)=d_p\,d_Q\log E(P, Q),$$
and the Bergman projective connection $S_B$ (see \cite{Fay}).
Denote by $S_\omega$ the projective connection defined via
$$\left\{\int^P\omega, x(P)\right\}\,,$$
where the braces denote the Schwarzian derivative.

The following theorem gives the variational formulas for the regularized determinant with respect to moduli.

\begin{thm}\label{VarTor}
Let $$\tilde Q=\frac{\det\Delta^{\tilde m}}{\det\Im {\mathbb B}},\quad Q=\frac{\det(\Delta, \mathring\Delta)}{\det\Im {\mathbb B}}\,.$$
Then the following variational formulas hold:
\begin{equation}\label{twistTor}
\frac{\partial \log Q}{\partial \theta_k}=\frac{\partial \log \tilde Q}{\partial \theta_k}=-\frac{1}{6\pi}\Re\big\{\oint_{\gamma_k}\frac{S_B-S_\omega}{\omega} \big\},\quad k=1, \dots, 3g+n-3;
\end{equation}

\begin{equation}\label{shiftTor}
\frac{\partial \log Q}{\partial h_k}=\frac{\partial \log \tilde Q}{\partial h_k}=\frac{1}{6\pi}\Re\left\{\oint_{b_k}\frac{S_B-S_\omega}{\omega} \right\},\quad k=1, \dots, g;
\end{equation}

\begin{equation}\label{stretchTor}
\frac{\partial \log Q}{\partial \tau_k}=\frac{\partial \log \tilde Q}{\partial \tau_k}=-\frac{1}{6\pi}\Im\left\{\oint_{\pm A_k \pm A'_k\mp C_k}
\frac{S_B-S_\omega}{\omega} \right\}.
\end{equation}
\end{thm}
Although methods of Fay's memoir \cite{Fay} allows to derive Theorem~\ref{VarTor} from Proposition~\ref{ResKer} in the same manner as \cite[Theorem 9]{KokKor} was derived from \cite[Proposition 2]{KokKor},
 we decided to use another approach, which is shorter and more transparent than that of~\cite{KokKor}.  We rely on the contour integral representation of the operator-zeta function and the variations of individual eigenvalues of the Laplacian $\Delta^{\tilde m}$, see Lemma~\ref{eigen} below.
\begin{lem}\label{eigen} Let $\lambda_j$ be an eigenvalue of $\Delta^{\tilde m}$ and let $\phi_j$ be the corresponding (real-valued) normalized eigenfunction.
The one-form
\begin{equation}
\Omega_j=(\partial_z \phi_j(z, \bar z))^2dz+\frac{\lambda_j}{4}\phi_j(z, \bar z)^2d\bar z
\end{equation}
is closed in the flat part of $({\cal M}, \tilde m)$ outside the conical singularities.
One has the following variational formulas:
\begin{equation}\label{eig1}
\frac{\partial \lambda_j}{\partial \theta_k}=4\Re\left\{\oint_{\gamma_k} \Omega_j
\right\},
\end{equation}
\begin{equation}
\frac{\partial \lambda_j}{\partial h_k}=-4\Re\left\{\oint_{b_k}\Omega_j
\right\},
\end{equation}
\begin{equation}
\frac{\partial \lambda_j}{\partial \tau_k}=4\Im\left\{\oint_{{\pm A_k \pm A'_k\mp C_k}} \Omega_j
\right\}.
\end{equation}
\end{lem}
\begin{proof}[Proof of Lemma~\ref{eigen}] All statements  immediately follow from Proposition \ref{ResKer} (cf. \cite{Fay}, p. 53, f-la 3.17) and the relation
$$\dot{\lambda}_n=\iint_{\cal M} \, {\rm Res}\left((\lambda-\lambda_n)\dot{G}(x, y; \lambda); \lambda=\lambda_n \right)\Big|_{y=x}d\tilde m(x)\,.$$
However, Lemma~\ref{eigen} can also be  proved independently. Let us omit the index $j$ and denote the eigenvalue by $\lambda$ and the corresponding   (real-valued) normalized eigenfunction by $\phi$.  For instance, to prove~\eqref{eig1} observe (cf. \eqref{diff1} and \eqref{diff2})
that the derivative $\dot{\phi}$ of  $\phi$ with respect to $\theta_k$ has the jump
$i(\phi_z-\phi_{\bar z})$ on the contour $\gamma_k$, whereas $\dot{\phi_z}$ has there the jump $i(\phi_{zz}-\phi_{z\bar z})$.
Denote by $\hat{{\cal M}}$ the surface ${\cal M}$ cut along the contour $\gamma_k$. We have
 $$\iint_{\hat{{\cal M}}}\phi \dot{\phi}=\frac{1}{\lambda}\iint_{\hat{{\cal M}}}\Delta^{\tilde{m}}\phi \dot{\phi}=\frac{1}{\lambda}\left\{
 2i\int_{\partial \hat{{\cal M}}}\phi_{\bar z}\dot{\phi}d\bar z+\phi \dot{\phi_z}dz+\iint_{\hat{{\cal M}}}\phi(\lambda \phi)^\cdot
 \right\}$$
$$=\frac{1}{\lambda}\left\{
-2\int_{\gamma_k}\phi_{\bar z}(\phi_z-\phi_{\bar z})d\bar z+\phi(\phi_{zz}-\phi_{z\bar z})dz+\dot{\lambda}+\lambda \iint_{\hat{{\cal M}}}\phi \dot{\phi}
\right\}$$
Now (\ref{eig1}) follows from the relations
$$\phi_{\bar z}\phi_zd\bar z-(\phi_{\bar z})^2d\bar z+\phi\phi_{zz}dz-\phi\phi_{z\bar z}dz=
d(\phi \phi_z)-(\phi_z)^2dz-\phi\phi_{z\bar z}d\bar z-(\phi_{\bar z})^2d\bar z-\phi\phi_{z\bar z}dz$$
and $$\phi_{z\bar z}=\frac{\lambda}{4}\phi;$$
the latter one, of course, holds only in the flat part of $({\cal M}, \tilde m)$.
\end{proof}

\begin{proof}[Proof of Theorem~\ref{VarTor}] From now on $\lambda$ stands for the spectral parameter (we assume that it is real and negative), $\{\lambda_k\}$ is the spectrum of $\Delta^{\tilde m}$, $z$ is the complex variable of integration
which at some points also becomes the spectral parameter (one of the the arguments of the resolvent kernel), $x$ and $y$ will denote the (flat)
complex local coordinates of
points near the contour $\gamma_k$.
 We start from the following integral representation of the zeta-function of the operator $\Delta^{\tilde m}-\lambda$ through the trace of the second
 power of the resolvent:
\begin{equation}
s\zeta(s+1; \Delta^{\tilde m}-\lambda)=\frac{1}{2\pi i}\int_{\Gamma_\lambda}(z-\lambda)^{-s}{\rm Tr}\left((\Delta^{\tilde m}-z)^{-2}\right)dz\,,
\end{equation}
where $\Gamma_\lambda$ is the contour connecting $-\infty+i\epsilon$ with $-\infty-i\epsilon$ and following the cut $(-\infty, \lambda)$ at the (sufficiently small) distance $\epsilon>0$.

Differentiating this formula with respect to $\theta_k$ (dot stands for such a derivative) and making use of (\ref{eig1}), we get
\begin{eqnarray}
s\dot{\zeta}(s+1, \Delta^{\tilde m}-\lambda)=-\frac{1}{\pi i}\int_{\Gamma_\lambda}(z-\lambda)^{-s}\sum_{\lambda_n>0}\frac{\dot{\lambda_n}}{(\lambda_n-z)^3}dz\nonumber\\
=-\frac{4}{\pi i}\int_{\Gamma_\lambda}(z-\lambda)^{-s}\sum_n\frac{\Re\left\{\oint_{\gamma_k}(\partial_x \phi_n(x, \bar x))^2dx+\frac{\lambda_n}{4}\phi_n(x, \bar x)^2
d\bar x\right\}} {(\lambda_n-z)^3}.\label{der1}
\end{eqnarray}

One can assume that the contour $\gamma_k$ is parallel to the imaginary axis and, therefore, $\Re\oint_{\gamma_k} \phi_n^2d\bar x=0$
(the latter trick does not work when one differentiates with respect to other moduli $h_k$, $\tau_k$, in these cases the proof gets a little bit longer) and
the right hand side of (\ref{der1}) can be rewritten as
\begin{equation}\label{der2}
-\frac{2}{\pi i}\oint_{\gamma_k}\int_{\Gamma_\lambda}(z-\lambda)^{-s}\sum_n\frac{(\partial_x \phi_n(x, \bar x))^2}{(\lambda_n-z)^3}dx  dz -\frac{2}{\pi i}\oint_{\gamma_k}\int_{\Gamma_\lambda}
 \sum_n \frac{(\partial_{\bar x} \phi_n(x, \bar x))^2}{(\lambda_n-z)^3}d\bar x dz\,.
\end{equation}
Using the standard resolvent kernel representation
$$G(x, y; z)=\sum_n\frac{\phi_n(x, \bar x)\phi_n(y, \bar y)}{\lambda_n-z}\,,$$
where summation is understood in the sense of the theory of distributions, one can easily identify the sum under the first (resp. the second) integral with
$\frac{1}{2}\left(\frac{d^2}{(dz)^2}G''_{xy}(x, y; z)\right)\Big|_{y=x}$  (resp. $\frac{1}{2}\left(\frac{d^2}{(dz)^2}G''_{\bar x \bar y}(x, y; z)\right)\Big|_{y=x}$.
It should be noted that although the resolvent kernel (as well as its second $xy$-derivative) is singular at the diagonal $x=y$, after
differentiation with respect the spectral parameter this singularity disappears. Using  Theorem 2.7 from Fay's memoir~\cite{Fay}, we get
 $$\left(\frac{d^2}{(dz)^2}G''_{xy}(x, y; z)\right)\Big|_{y=x}$$
 $$
 =\frac{d^2}{(dz)^2}\left[\left(G''_{xy}(x, y; z)-\frac{1}{4\pi}\frac{1}{(x-y)^2}+\frac{z}{16\pi}\frac{\bar y-\bar x}{y-x}\right)\Big|_{y=x}\right]=:
 \frac{d^2}{(dz)^2}\Phi(x, z);
$$
note that in
 \cite[(2.32)]{Fay} one should take  $r=|x-y|$,  $H_0=1$, and $H_1=0$ as the metric $\tilde m$ is flat in a vicinity of the contour $\gamma_k$.
 Clearly, in the right hand side of (\ref{der1}) the sum under the second integral equals to
 $$\frac{d^2}{(dz)^2}\overline {\Phi(x, \bar z)}.$$

Integration by parts in (\ref{der1}) (and the change of variable $s+1\mapsto s$) leads to
$$
\dot{\zeta}(s; \Delta^{\tilde m}-\lambda)=-\frac{1}{\pi i}\oint_{\gamma_k}\int_{\Gamma_\lambda} (z-\lambda)^{-s}\left[\frac{d}{dz}\Phi(x, z)dx+
\frac{d}{dz}\overline{\Phi(x,\bar z)}d\bar x\right]dz.
$$
Shrinking the contour $\Gamma_\lambda$ to the half-line $(-\infty, \lambda)$, we obtain
\begin{equation}\label{der3}
\dot{\zeta}(s, \Delta^{\tilde m}-\lambda)=-\frac{2\sin(\pi s)}{\pi}\int_{-\infty}^\lambda \oint_{\gamma_k} (\lambda-t)^{-s}\left[\frac{d}{dt}\Phi(x, t)dx+
\frac{d}{dt}\overline{\Phi(x, t)}d\bar x\right]dt.
\end{equation}
We differentiate \eqref{der3} with respect to $s$ and set $s=0$ and $\lambda=0$. As a result we get
\begin{equation}\label{ura1}
\dot{\zeta}'(0, \Delta^{\tilde m})= -2\oint_{\gamma_k}  \Phi(x, t)\Big|_{t=-\infty}^{t=0}dx+\overline{\Phi(x, t)
}\Big|_{t=-\infty}^0d\bar x
\end{equation}
\begin{equation}\label{ura2}
=-2\oint_{\gamma_k}\left(\frac{1}{24\pi}S_B(x)-\frac{1}{4}\sum_{\alpha, \beta=1}^g(\Im {\mathbb B})^{-1}_{\alpha \beta}v_\alpha(x)v_\beta(x)
\right)dx+\overline {\Big( \cdots \Big)} d\bar x
\end{equation}
$$=-\frac{1}{6\pi}\Re\left\{
\oint_{\gamma_k}S_B(x)dx-6\pi\oint_{\gamma_k}\sum_{\alpha, \beta=1}^g(\Im {\mathbb B})^{-1}_{\alpha \beta}v_\alpha(x)v_\beta(x)dx
\right\},$$
which is the same as (\ref{twistTor}). To pass from (\ref{ura1}) to (\ref{ura2}) we used the classical Lemma~\ref{kok-2} given below (cf. \cite{Fay}, p. 30).
\end{proof}
\begin{lem} \label{kok-2} Let, as before, $G(x, y; \lambda)$ be the resolvent kernel for the operator $\Delta^{\tilde{m}}$.
Define the Green function $G(x, y)$ of the operator $\Delta^{{\tilde {m}}}$ via the expansion
\begin{equation}\label{greenFr}
G(x, y; \lambda)=-\frac{1}{{\rm Area}({\cal M}, \tilde{m})}\frac{1}{\lambda}+G(x, y)+O(\lambda),\quad\lambda\to 0.\end{equation}
Then $G''_{xy}(\, \cdot\, , \,\cdot\,)$ is a meromorphic bidifferential with  unique double pole at the diagonal $x=y$, related to the
Bergman bidifferential $W(x, y)$ via
$$4\pi G''_{xy}(x, y)=W(x, y)-\pi \sum_1^g\Im{\mathbb B}_{ij}^{-1}v_i(x)v_j(y)\,,$$
where $v_1, \dots, v_g$ are the normalized holomorphic differentials on the compact Riemann surface ${\cal M}$.
In particular, we have
$$\left[4\pi G''_{xy}(x, y)-\frac{1}{(x-y)^2}\right]\Big|_{y=x}=\frac{1}{6}S_B(x)- \pi \sum_1^g\Im{\mathbb B}_{ij}^{-1}v_i(x)v_j(x)\,,$$
where $S_B$ is the Bergman projective connection.
\end{lem}
\begin{proof} Clearly, the Green function (symmetric with respect to its both arguments) is the (unique) solution to the problem
$$\begin{cases}
\Delta^{\tilde{m}}_xG(x, y)=-\frac{1}{{\rm Area}({\cal M}, \tilde{m})}\ \ \text{for}\ \ x\neq y,\\
G(x, y)\sim \frac{1}{2\pi}\log|x-y| \ \ \text{as} \ \ x\to y.
\end{cases}$$
Thus, $\partial_{\bar x}G''_{xy}=0$ for $x\neq y$ and $4\pi G''_{xy}(x, y)=\frac{1}{(x-y)^2}+O(1)$ as $y\to x$.
This implies the equation
\begin{equation}\label{neopr}
4\pi G''_{xy}(x,y)=W(x, y)+\sum_{i,j=1}^gc_{ij}\,v_i(x)v_j(y)
\end{equation}
with some constants $c_{ij}$.
Using Stokes theorem, it is easy to show that
\begin{equation}\label{ort}
v. p. \iint_{\cal M}G''_{xy}(x, y)\overline{v_i(x)}=0,\quad i=1, \dots, g.
\end{equation}
Plugging (\ref{neopr}) in the orthogonality conditions (\ref{ort}) and using Stokes theorem once again, one gets
the relations
$$c_{ij}=-\pi(\Im {\mathbb B})_{ij}^{-1}\,,\quad i,j=1, \dots, g.$$
\end{proof}

\begin{Remark}  For other moduli ($h_k$ and $\tau_k$) the trick with choosing the contour of integration parallel to the imaginary axis is impossible and
one has to work with the additional term
$$\sum\frac{\lambda_n}{4}\frac{(\phi_n(x, \bar x))^2}{(z-\lambda_n)^3}=\frac{1}{2}\left(\frac{d^2}{(dz)^2}G''_{x\bar x}(x, y; z)\right)\Big|_{y=x}\,.$$
To interchange the differentiation with respect to $z$ and pass to the limit $y\to x$ one should make use of the following corollary of \cite[Theorem 2.7]{Fay}:
$$\left(\frac{d^2}{(dz)^2}G''_{x\bar x}(x, y; z)\right)\Big|_{y=x}
$$
$$=\frac{d^2}{(dz)^2}\left\{\left(G''_{x\bar x}(x, y; z)-\frac{1}{16\pi}z\log|x-y|^2-\frac{1}{8\pi}
z(\log\frac{1}{2}\sqrt{z}+\gamma-1)\right)\Big|_{y=x}\right\}\,.
$$
After the same operations as before this term gives rise to the expressions
$$\Re\left[\frac{1}{{\rm Area}({\cal M}, \tilde{m})}\oint_{b_k} d\bar x\right]\quad\text{ or }\quad  \frac{1}{{\rm Area}({\cal M}, \tilde{m})}\oint_{{\pm A_k \pm A'_k\mp C_k}}d\bar x.$$
Both of them vanish.
\end{Remark}

\begin{Remark}
For non Friedrichs self-adjoint extensions of the Laplacian  on $({\cal M},\tilde{m})$  $\lambda=0$ is not an eigenvalue and (\ref{greenFr}) is no longer true. Determinants of such extensions were studied in \cite{HillKok}.
\end{Remark}
\section{Bergman tau-function on Mandelstam diagrams and explicit formulas for regularized determinant}

In this section we show that a solution
to the system of equations in partial derivatives (\ref{twistTor}, \ref{shiftTor}, \ref{stretchTor}) can be
found explicitly in terms of certain canonical objects related to the underlying Riemann surface ${\cal M}$ (theta-functions, prime-forms) and the divisor of the
meromorphic differential $\omega$. This leads to an explicit formula for the regularized determinant $\det(\Delta, \mathring\Delta)$ (up to moduli independent
multiplicative constant).

We construct the above mentioned solution as the modulus square of the function $\tau$ defined on the space of Mandelstam diagrams of a given genus.
(More precisely, only some integer power of $\tau$ is single-valued on the space of diagrams, the function $\tau$ itself is defined up to a unitary factor.)

 We start  with definition of the function $\tau$. Note that it is a straightforward generalization
 of the Bergman tau-function on the moduli space of Abelian differentials~\cite{KokKor} (i. e. the moduli space of pairs $(X, \omega)$, where $X$ is a compact Riemann surface, and $\omega$ is a holomorphic one-form on $X$) to the case of a meromorphic one-form $\omega$ with pure imaginary periods and simple poles with (fixed) real residues. This generalization
 (along with many others) was also recently discussed in~\cite{KorKalla}.

 The cases of genus $g=0$, $g=1$ and $g\geq 2$ should be considered separately, the first two are pretty elementary and do not involve somewhat complicated auxiliary objects.

  {\bf Genus zero case.} Let the Riemann surface $\cal M$ have genus zero. In this case the Mandelstam diagram ${\Pi}$ has no interior slits.
  The Riemann surface ${\cal M}$ is biholomorphically equivalent to the Riemann sphere ${\mathbb C}P^1$, let $z$ be the uniformizing parameter which came
  from ${\mathbb C}={\mathbb C}P^1\setminus{\infty}$.
  The canonical meromorphic bidifferential is given by
    $$
    W(P, Q)=\frac{dz(P)\,dz(Q)}{(z(P)-z(Q))^2}.
    $$
 Assume that the circles $O_1,\dots, O_{n_-}$ correspond to the left cylindrical ends of ${\cal M}$, i.e. $\cup_{1\leq\ell\leq n_-} O_\ell$ is the cross-section $\{p\in\mathcal M: x=-R\}$. Then there are  $n_+=n-n_-$ circles  $O_{n_-+1},\dots, O_{n}$ corresponding to the right cylindrical ends of ${\cal M}$.
 Let $P_k^-$ with $ k=1,\dots, n_-$ and $P_j^+$ with $j=1,\dots,n_+$
  be the corresponding points at infinity of the diagram ${\cal M}$ or, equivalently, the poles of the meromorphic differential $\omega$ with
  residues $-\frac{|O_k|}{2\pi}$ and $\frac{|O_{j+n_-}|}{2\pi}$ respectively.
 Let also $R_1, \dots, R_{n-2}$ be the zeros of the meromorphic differential $\omega$ or, equivalently, the end points of the
 semi-infinite slits of the diagram ${\cal M}$.

 Introduce the local parameters
 \begin{equation}\label{distpole} \zeta_k^-=\exp(2\pi{z}/{|O_k|})\quad \text{(resp. } \zeta_j^+=\exp(-2\pi z/|O_{n_-+j}|)\text{)} \end{equation}
 in vicinities of the poles  $P_k^-$ (resp. $P_{j}^+$) of the differential $\omega$
 and
 \begin{equation}\label{distzero}
 \zeta_\ell=\sqrt{z-z(R_l)}
 \end{equation}
in vicinities of the zeros $R_l$ of $\omega$.
We call the parameters \eqref{distpole}, \eqref{distzero} {\it distinguished}.
In what follows we denote by $W(R_l, \, \cdot\, )$ the meromorphic one-form on the Riemann surface ${\cal M}$
$$\frac{W(P,\, \cdot\,)}{d\zeta(P)}\Big|_{P=R_l},$$
where $\zeta$ is the distinguished local parameter in a vicinity of $R_l$; the quantities $W(P_k^\pm, \, \cdot\,)$ have similar meaning.
Introduce the function $\tau$ on the space of Mandelstam diagrams via
\begin{equation}\label{GENUSZERO}
\tau^{12}=\frac{1}{\omega^2(\,\cdot\,)}\frac{\prod_{k=1}^{n_-}W(P_k^-, \, \cdot\,)\prod_{j=1}^{n_+}W(P_j^+,
 \, \cdot \,)}{\prod_{l=1}^{n-2}W(R_l, \, \cdot\,)}.\end{equation}

Clearly, the right hand side of \eqref{GENUSZERO} is a holomorphic function on the Riemann surface ${\cal M}$  and, therefore,  a constant (depending on moduli).

{\bf Genus one case.}  Let the Riemann surface $\cal M$ have genus one. In this case the Mandelstam diagram ${\cal M}$ has one interior slit and the number of poles of the differential $\omega$ (i. e. the points at infinity of the diagram ${\cal M}$) equals to the number of zeros of
$\omega$ (i. e. the endpoints of the slits of the diagram ${\cal M}$).  For the poles and zeros we keep the same notation $P_k^\pm, R_l$ as before.
Let ${\mathbb B}$ be the $b$-period of the normalized $\left(\int_av=1\right)$ differential $v$ on the marked Riemann surface $({\cal M}, \{a, b\})$.
Let
$$v(R_l)=\frac{v(P)}{d\zeta(P)}\Big|_{P=R_l},$$
where $\zeta$ is the distinguished local parameter near $R_l$. The quantities $v(P_k^\pm)$ are defined similarly.
Define the function $\tau$ via
\begin{equation}\label{GENUSONE}
\tau^{12}=\left[\Theta_1'(0\,|\,{\mathbb B}\,)\right]^8\frac{\prod_{k=1}^{n_-}v(P_k^-)\prod_{j=1}^{n_+}v(P_j^+)}{\prod_{l=1}^{n}v(R_l)},
\end{equation}
where $\Theta_1$ is the first Jacobi's  theta-function.

{\bf Case of genus $g\geq 2$.} Let the Riemann surface ${\cal M}$ have genus $g\geq 2$. Following \cite{Fay}, introduce the (multivalued) $g(g-1)/2$-differential
$${\cal} C(P)=\frac{1}{{\cal W}[v_1, v_2, \dots, v_g](P)}\sum_{\alpha_1, \dots, \alpha_g=1}^g\frac{\partial^g\Theta(K^P)}{\partial z_{\alpha_1}\dots\partial
z_{\alpha_g}}v_{\alpha_1}\dots v_{\alpha_g}(P)\,,$$
where $\{v_1, \dots, v_g\}$ is the normalized basis of holomorphic differentials on ${\cal M}$, ${\cal W}$ is the Wronskian determinant
of the holomorphic differentials, $K^P$ is the vector of Riemann constants, $\Theta$ is the theta-function built from the matrix ${\mathbb B}$ of
the $b$-periods of the Riemann surface ${\cal M}$. Let $E(P, Q)$ be the prime-form on ${\cal M}$ (see \cite{Fay}).

It is convenient to denote the zeros and poles of the meromorphic one-form $\omega$ by $D_l$. The divisor of the one-form
$\omega$ can be written as
$$(\omega)=\sum_{l}d_l D_l,$$
where $d_l=1$ if $D_l$ is a zero and $d_l=-1$ if $D_l$ is a pole of $\omega$.

Define the function $\tau$ via
\begin{equation}\label{HIGHERGENUS}
\tau={\cal F}^{2/3}e^{-\frac{\pi i}{6}<{\bf r}, {\mathbb B}{\bf r}>}\prod_{m<n}\left\{E(D_m, D_n)\right\}^{d_md_n/6}\,,
\end{equation}
where the (scalar)
$${\cal F}=[\omega(P)]^{(g-1)/2}e^{-\pi i <{\bf r}, K^P>}\left\{\prod_{m}[E(P, D_m)]^{\frac{(1-g)d_m}{2}}\right\}{\cal C}(P)$$
is independent of the point $P$ of the Riemann surface ${\cal M}$ and the integer vector ${\bf r}$ is defined by the equality
$${\cal A}((\omega))+2K^P+{\mathbb B}{\bf r}+{\bf q}=0;$$
here ${\bf q}$ is another integer vector and  the initial point of the Abel map ${\cal A}$ coincides with $P$. If one argument (or both)
of the prime-form coincides with some point $D_l$ then the prime-form is computed with respect to the distinguished local parameter at this point.

\begin{Remark}\label{shiftinv} If $n_-=n_+$ and there is a one-to-one correspondence between the sets $\{O_k\}_{k=1}^{n_-}$ and $\{O_j\}_{j=n_-+1}^{n}$, then as $\mathring{\mathcal M}$ one take the union $\cup_{\ell=1}^{n_-} \Bbb R\times O_\ell$ of $n/2$ infinite cylinders. As a result the right hand sides of \eqref{GENUSZERO}, \eqref{GENUSONE}, and~\eqref{HIGHERGENUS} turns out to be invariant under the horizontal shifts of the diagram $z\mapsto z+C$ or, what is the same, independent of the choice of the initial moment of time $\tau_0=0$.
\end{Remark}

The following theorem states that the logarithm of the modulus square of the just introduced function $\tau$ has the same derivatives with respect to moduli as the
quantity $\log \frac{\det(\Delta, \mathring\Delta)}{\det \Im {\mathbb B}}$.

\begin{thm}\label{govern}
Then the following variational formulas hold:
\begin{equation}\label{twisttau}
\frac{\partial \log |\tau|^2}{\partial \theta_k}=-\frac{1}{6\pi}\Re\big\{\oint_{\gamma_k}\frac{S_B-S_\omega}{\omega} \big\},\quad k=1, \dots, 3g+n-3;
\end{equation}
\begin{equation}\label{shifttau}
\frac{\partial \log |\tau|^2}{\partial h_k}=\frac{1}{6\pi}\Re\left\{\oint_{b_k}\frac{S_B-S_\omega}{\omega} \right\},\quad k=1, \dots, g;
\end{equation}
\begin{equation}\label{stretchtau}
\frac{\partial \log |\tau|^2}{\partial \tau_k}=-\frac{1}{6\pi}\Im\left\{\oint_{\pm A_k \pm A'_k\mp C_k}
\frac{S_B-S_\omega}{\omega} \right\}.
\end{equation}
\end{thm}

\begin{proof} The proof  is completely similar to the proofs of~\cite[Theorems 6 and 7]{KokKor}. First, one has to derive variational formulas (under variations of the moduli $\theta_k, \tau_k$, and $h_k$) for the basic objects on the compact Riemann surface ${\cal M}$ which appear as ingredients in the explicit expression for the function $\tau$  (i.e. the basic holomorphic differentials, the matrix of $b$-periods, the canonical meromorphic bidifferential, the prime-form, the vector of the Riemann constants, and the multi-valued differential ${\cal C}$). Then one has to check \eqref{twisttau}--\eqref{stretchtau} via direct calculation. We decided not to repeat this rather long calculation here, we only sketch the proof in a relatively simple case of a low genus curve, where most of the technicalities disappear. For instance, let us prove~\eqref{twisttau} in case $g=1$.

 Choose a canonical basis $\{a, b\}$ of cycles on ${\cal M}$ and introduce the normalized holomorphic differential $v$ such that
 $$\int_av=1\ \ \ {\rm and}\ \ \ \int_bv={\mathbb B}\,.$$

 Take $P\in \Pi$, then in vicinity of $P\in {\cal M}$ the ratio of the two one-forms $\frac{v}{dz}$ defines a scalar function. Denote the value of this function at $P$ by $v(P)$. For a fixed $P$ this value still depends on the moduli $\theta_k, h_k, \tau_k$.  Using the same idea as in the proof of Proposition~\ref{ResKer} (see also \cite[Proof of Theorem 3]{KokKor}),
 one can prove the following variational formula for the $v(P)$ with respect to the coordinate $\theta_k$:
 \begin{equation}\label{rauch1}
 \frac{\partial v(P)}{\partial \theta_k}=\frac{1}{2\pi}\int_{\gamma_k}\frac{ W(\,\cdot\,, P)v}{\omega}\,,
  \end{equation}
where $W$ is the Bergman bidifferential and the one form $W(\,\cdot\,, P)$  is defined as $\frac{W(\,\cdot, \, Q)}{dz(Q)}\Big|_{Q=P}$.
Integrating \eqref{rauch1} over the $b$-cycle, one gets the following variational formula for the $b$-period:
\begin{equation}\label{b-rauch}
\frac{\partial {\mathbb B}}{\partial \theta_k}=i\int_{\gamma_k}\frac{v^2}{\omega}\,.
\end{equation}
Moreover, since the distinguished local parameters~\eqref{distpole} at $P^-_k$, $P^+_j$ and~\eqref{distzero} at $R_l$ are moduli independent, \eqref{rauch1} implies that
\begin{equation}\label{rauch2}
 \frac{\partial v(P_k^-)}{\partial \theta_k}=\frac{1}{2\pi}\int_{\gamma_k}\frac{ W(\,\cdot\,, P_k^-)v}{\omega}\,,
  \end{equation}
\begin{equation}\label{rauch3}
 \frac{\partial v(P_j^+)}{\partial \theta_k}=\frac{1}{2\pi}\int_{\gamma_k}\frac{W(\,\cdot\,, P_j^+)v}{\omega}\,,
  \end{equation}
  and
\begin{equation}\label{rauch4}
 \frac{\partial v(R_l)}{\partial \theta_k}=\frac{1}{2\pi}\int_{\gamma_k}\frac{ W(\,\cdot\,, R_l)v}{\omega}\,,
  \end{equation}
where, say, $v(P_k^-)=\frac{v}{d\zeta^-_k}\Big|_{P_k}$ and $W(\cdot, R_l)=\frac{W(\,\cdot\,, Q)}{d\zeta_l(Q)}\Big|_{Q=R_l}$, etc.

Now using \eqref{rauch2}--\eqref{rauch4}) and the well-known formula
$$W(z_1, z_2)=\left[\frak{P}(\int_{z_1}^{z_2}v)-\frac{4 i\pi}{3}\frac{d}{d{\mathbb B}}\log\Theta_1'(0\,|\,{\mathbb B})\right]v(z_1)v(z_2)\,$$
for the Bergman bidifferential on an elliptic curve (see, e.g., \cite{oldFay}), where $\frak{P}$ is the Weierstrass $\frak{P}$-function, we arrive at
$$
\partial_{\theta_k}\log \frac{\prod_{k=1}^{n_-}v(P_k^-)\prod_{j=1}^{n_+}v(P_j^+)}{\prod_{l=1}^{n}v(R_l)}$$
$$
=\frac{1}{2\pi}\int_{\gamma_k}\frac{v(Q)}{\omega(Q)}\left\{\sum_{k=1}^{n_-}\frac{W(Q, P_k^-)}{v(P_k^-)}+ \sum_{j=1}^{n_+}\frac{W(Q, P_j^+)}{v(P_j^+)}-
\sum_{l=1}^{n}\frac{W(Q, R_l)}{v(R_l)}\right\}$$
\begin{equation}\label{Weier}
=\frac{1}{2\pi}\int_{\gamma_k}\frac{v^2(Q)}{\omega(Q)}\left[\sum_{k=1}^{n_-}\frak{P}(\int_{P^-_k}^Qv)+  \sum_{j=1}^{n_+}\frak{P}(\int_{P^+_j}^Qv)-
\sum_{l=1}^{n}\frak{P}(\int_{R_l}^Qv)\right]\,.\end{equation}

Consider the meromorphic function $R'=\frac{\omega}{v}$ on ${\cal M}$. (Clearly, it can be considered as the derivative of the (mulivalued) map $\xi=\int^Pv\mapsto R(\xi)=\int^P\omega$). Observe that the expression in the square brackets in \eqref{Weier} coincides with
$$\frac{d}{d\xi}\left(\frac{R''(\xi)}{R'(\xi)}\right)\,.$$
Therefore, using integrating by parts, \eqref{Weier} can be rewritten as
\begin{eqnarray}
\frac{1}{2\pi}\int_{\gamma_k}\frac{1}{R'(\xi)} \frac{d}{d\xi}\left(\frac{R''(\xi)}{R'(\xi)}\right)d\xi=\frac{1}{2\pi}\int_{\gamma_k}
\frac{(R''(\xi))^2}{(R'(\xi))^3}d\xi=\frac{1}{\pi}\int_{\gamma_k}\frac{\{R, \xi\}}{R'(\xi)}d\xi\\
 =\frac{1}{\pi}\int_{\gamma_k}\frac{\{\int^P\omega,\, \cdot\,  \}-\{\int^Pv, \,\cdot\, \}}{\omega}\label{diffS},
\end{eqnarray}
where $\{\,\cdot\,, \,\cdot\,\}$ denotes the Schwarzian derivative. The integrand in~\eqref{diffS} is a meromorphic one-form: the ratio of the difference of two projective connections (this difference gives a quadratic differential) and a meromorphic one-form.

Moreover, using \eqref{b-rauch}, we get
$$ \partial_{\theta_k}\log\left[\Theta_1'(0\,|\,{\mathbb B}\,)\right]^8=8i \frac{\partial \log \Theta_1'(0\,|\,{\mathbb B}\,)}{\partial {\mathbb B}}                    \int_{\gamma_k}\frac{v^2}{\omega}$$
and, therefore,
\begin{equation} \label{connect1}
\partial_{\theta_k}\log(\tau^{12})=\frac{1}{\pi}\int_{\gamma_k}\frac{\{\int^P\omega, \,\cdot\,\}-\left[\{\int^Pv, \,\cdot\,\}-8i\pi\frac{\partial \log \Theta_1'(0\,|\,{\mathbb B}\,)}{\partial {\mathbb B}}v^2\right]    }{\omega}\,,\end{equation}
where $\tau$ is from \eqref{GENUSONE}. It is known (see, e. g., \cite{oldFay}) that the expression in square brackets in \eqref{connect1} coincides with the Bergman projective connection. Therefore
$$\partial_{\theta_k}\log\tau=-\frac{1}{12\pi}\int_{\gamma_k}\frac{S_B-S_\omega}{\omega},$$
which proves~\eqref{twisttau}.
\end{proof}

The following immediate corollary of Theorem~\ref{govern} is the main result of the present paper.

\begin{cor} One has the following explicit expression for the regularized determinant of the Laplacian on the Mandelstam diagram ${\cal M}$:
\begin{equation}\label{result}
\det(\Delta, \mathring\Delta)=C\, \det\Im {\mathbb B}\,|\tau|^2,
\end{equation}
where $C$ is moduli independent constant.
\end{cor}

\begin{Remark}
If the unperturbed diagram $\mathring{\cal M}$ is a disjoint union of infinite cylinders then the regularized determinant $\det(\Delta, \mathring\Delta)$
is invariant with respect to horizontal shifts of the diagram ${\cal M}$ (i. e. the choice of the initial moment of time $\tau_0$).
The same is, of course, true for
the right hand side of \eqref{result}, cf. Remark~\ref{shiftinv}.
\end{Remark}

\appendix
\section{Appendix.}
\subsection{Proof of Lemmas~\ref{fr} and~\ref{l2}}
\begin{proof}[Proof of Lemma~\ref{fr}]
 The proof is based on well-known methods of the theory  of elliptic
 boundary value
problems, see e.g.~\cite{KM,KMR,Lions Magenes}. Recall that a bounded operator is
said to be Fredholm if its kernel and cokernel are finite
dimensional and its range is closed. We will rely on the following
lemma  due to Peetre, see e.g.~\cite[Lemma~3.4.1]{KMR} or
\cite[Lemma~5.1]{Lions Magenes}:
\begin{itemize}
\item[]{\it Let $\mathcal X,\mathcal Y$ and $\mathcal Z$ be Hilbert spaces, where $\mathcal X$ is compactly embedded into $\mathcal Z$. Furthemore, let $\mathcal L$ be a linear continuous operator from $\mathcal X$ to $\mathcal Y$. Then the next two assertions are equivalent: (i) the range of $\mathcal L$ is closed in $\mathcal Y$ and $\dim \ker \mathcal L<\infty$, (ii) there exists a constant $C$, such that
    \begin{equation}\label{1}
    \|u;{\mathcal X}\|\leq C(\|\mathcal L u;{\mathcal Y}\|+\|u;\mathcal Z\|)\quad \forall u\in \mathcal X.
    \end{equation}}
\end{itemize}

 Below we assume that
\begin{equation}\label{cond}
\{\mu^2-(\xi+i\epsilon)^2\in\Bbb C: \xi\in\Bbb R\}\cap\{0,4\pi^2
\ell^2 |O_k|^{-2}:\ell\in\Bbb N, 1\leq k\leq n\}=\varnothing
\end{equation}
  and
establish the estimate
 \begin{equation}\label{peetre}
\|u;\mathcal D_\epsilon\|\leq C(\|(\Delta_\epsilon-\mu^2)u;
L^2_\epsilon(\mathcal M)\|+\| u; L^2(\mathcal M)\|)
\end{equation}
 of type~\eqref{1}. 

Let $R>0$ be  so large that there are no conical points on $\mathcal
M$ with coordinate $x\notin (-R,R)$. Take some smooth functions
$\varrho_k$, $1\leq k\leq n$, on $\mathcal M$ satisfying
\begin{equation*}
\varrho_k(p)=\left\{
              \begin{array}{ll}
                1, &  p\in(-\infty,-R-1)\times O_k \ (  \text {resp.} p\in(R+1,\infty)\times O_k) , \\
                0, &  p\in \mathcal M\setminus (-\infty,-R)\times O_k \ (  \text {resp. } p\in\mathcal M\setminus (R,\infty)\times O_k),
              \end{array}
            \right.
\end{equation*}
if $O_k$ is the cross-section of a cylindrical end directed to the
left (resp. directed to the right). We also set
$\varrho_0=1-\sum_k\varrho_k $, then $\{\varrho_k\}_{k=0}^{n}$ is a
partition of unity on $\mathcal M$.

Let $L^2_\epsilon(\Bbb R\times O_k)$ be the weighted space with the
norm $\bigl(\int_{\Bbb R\times O_k}|e^{\gamma_k x}\mathsf u(x,y)|^2
\,dx\,dy\bigr)^{1/2}$, where $\gamma_k=-\epsilon$ if the
corresponding cylindrical end of  $\mathcal M$ is directed to the
left and $\gamma_k=\epsilon$ if the end is directed to the right.
Introduce the weighted Sobolev space $H^2_\epsilon(\Bbb R\times
O_k)$ as completion of the set $C_c^\infty(\Bbb R\times O_k)$ in the
norm
$$
\|\mathsf u;H^2_\epsilon(\mathbb R\times O_k)\|=\Bigl(\sum_{ p+q\leq
2 } \|\partial_x^p\partial _y^q \mathsf u; L^2_\epsilon(\Bbb R\times
O_k)\|^2\Bigr)^{1/2}.
$$
For $u\in\mathcal D_\epsilon\bigl(\subset H^1_\epsilon (\mathcal
M)\bigr)$ we extend $\mathsf u_k=\varrho_k u$, $1\leq k\leq n$, to
$\Bbb R\times O_k$ by zero. Clearly, the right hand side of the
equation
\begin{equation}\label{cyl}
(-\partial_x^2+\Delta_{O_k}-\mu^2)\mathsf u_k=\mathsf f_k
\end{equation}
 is in
$L^2_\epsilon(\Bbb R\times O_k)$.  Applying the Fourier-Laplace
transform $\mathscr F_{x\mapsto \xi+i\gamma_k}$ we pass
from~\eqref{cyl} to the equation
\begin{equation}\label{mo}
(\Delta_{O_k}-\mu^2+(\xi+i\gamma_k)^2)\mathscr F_{x\mapsto
\xi+i\gamma_k}{\mathsf u}_k=\mathscr F_{x\mapsto
\xi+i\gamma_k}{\mathsf f}_k,\quad \xi\in\Bbb R.
\end{equation}
 The norm of the inverse of the operator $\Delta_{O_k}-\mu^2+(\xi+i\gamma_k)^2$ in $L^2(O_k)$ is
 bounded by the reciprocal of the distance from the
parabola $$\{\mu^2-(\xi+i\gamma_k)^2:\xi\in \Bbb
R\}=\{\mu^2-(\xi+i\epsilon)^2:\xi\in \Bbb R\}$$ to the spectrum
$\{0,4\pi^2 \ell^2 |O_k|^{-2}:\ell\in\Bbb N\}$ of the selfadjoint
Laplacian $\Delta_{O_k}$ on $O_k$, cf.~\eqref{cond}.  This together
with elliptic coercive estimates for $\Delta_{O_k}$ and the Parseval
equality implies
  \begin{equation}\label{coer}
\begin{aligned}
 \|\mathsf u_k; H^2_\epsilon(\mathbb R\times O_k)\| \leq  \mathrm{c} \bigl\|\bigl(-\partial_x^2+\Delta_{O_k} -\mu^2\bigr) \mathsf u_k; L^2_\epsilon(\mathbb R\times O_k)\bigr\|
\end{aligned}
\end{equation}
with an independent of $\mathsf u\in H^2_\epsilon(\mathbb R\times
O_k)$ constant $\mathrm{c}$; moreover, the operator
$$
-\partial_x^2+\Delta_{O_k}-\mu^2:H^2_\epsilon(\mathbb R\times
O_k)\to L^2_\epsilon(\mathbb R\times O_k)
$$
yields an isomorphism, see e.g.~\cite[Chapter 5]{KMR} or~\cite{KM} for details.

From~\eqref{norm} and~\eqref{coer} it immediately follows that
\begin{equation}\label{+1}
\begin{aligned}
\|u;\mathcal D_\epsilon\|\leq \sum_{k=0}^{n}\|\varrho_k u;\mathcal
D_\epsilon\|&\leq \|(\Delta_\epsilon -\mu^2)\varrho_0 u; L^2_\epsilon (\mathcal
M)\|\\&+(1+|\mu|^2)\|\varrho_0 u; L^2_\epsilon(\mathcal
M)\|+\sum_{k=1}^{n}\|\varrho_k u; H^2_\epsilon(\Bbb R\times O_k)\|
\\
\leq \|(\Delta_\epsilon -\mu^2) u; L^2_\epsilon (\mathcal
M)\|&+\sum_{k=0}^{n}\|[\varrho_k,\Delta_\epsilon] u; L^2_\epsilon
(\mathcal M)\|+(1+|\mu|^2)\|\varrho_0 u; L^2_\epsilon(\mathcal M)\|.
\end{aligned}
\end{equation}
Here the commutators $[\varrho_k,\Delta_\epsilon]$ are first order
differential operators with smooth coefficients supported
 on a smooth compact part of $\mathcal M$. Local
 elliptic coercive estimates imply
\begin{equation}\label{+2}
\|[\varrho_k,\Delta_\epsilon] u; L^2(\mathcal M)\|\leq
C\bigl(\|\eta(\Delta_\epsilon -\mu^2) u; L^2 (\mathcal M)\|+\|\eta
u; L^2(\mathcal M)\|\bigr),
\end{equation}
where $\eta\in C_c^\infty(\mathcal M)$ is such that  $\eta
[\varrho_k,\Delta_\epsilon]=[\varrho_k,\Delta_\epsilon]$ and
$\eta\varrho_0=\varrho_0$. Now the estimate~\eqref{peetre} follows from~\eqref{+1} and~\eqref{+2}.
It remains to note that compactness of the embedding $\mathcal
D_\epsilon\hookrightarrow L^2(\mathcal M)$ is a
consequence of the compactness of
$$
H^2_\epsilon (\Bbb R\times O_k)\ni \varrho_k  u\mapsto \varrho_k
 u\in L^2(\Bbb R\times O_k),\quad \mathfrak D\ni \varrho_0 u\mapsto \varrho_0 u\in L^2(\mathcal
 M),
$$
where   the domain $\mathfrak D$ of the selfadjoint Friedrichs
extension of Dirichlet Laplacian on $\mathcal M_{R}=\{p\in\mathcal M:|x|\leq R\}$ is compactly embedded into $L^2(\mathcal M_{R})$.

The above argument also implies that the graph norm~\eqref{norm} in
$\mathcal D_\epsilon$ is equivalent to the norm
\begin{equation}\label{norm'}
\|u;\mathcal D_\epsilon\|\asymp\|\varrho_0u; \mathfrak
D\|+\sum_{k=1}^{n}\|\varrho_k u; H^2_\epsilon(\Bbb R\times O_k)\|,
\end{equation}
and the space $\mathcal D_\epsilon$ consists of all elements $u\in
H^1_\epsilon(\mathcal M)$ with finite norm~\eqref{norm'}.

  In order to see that the cokernel of the operator~\eqref{ep} is
finite-dimensional, one can apply a similar argument to the
 adjoint m-sectorial operator $(\Delta_\epsilon)^*$ in
$L^2_\epsilon(\mathcal M)$. In particular, it turns out that the
graph norm of $(\Delta_\epsilon)^*$ is equivalent to the
norm~\eqref{norm'} and $\mathcal D_\epsilon$ is the domain of
$(\Delta_\epsilon)^*$.

We have proved that the operator~\eqref{ep} is Fredholm
if~\eqref{cond} holds true. Now we assume that for some $\xi\in\Bbb
R$ the number $\mu^2-(\xi+i\epsilon)^2$ coincides with an eigenvalue
$\lambda$ of $\Delta_{O_k}$ and show that the operator~\eqref{ep} is
not Fredholm.

For instance, let $O_k$ correspond to  a cylindrical end directed to
the right. Introduce a cutoff function $\chi\in C^\infty(\Bbb R)$
such that $\chi(x)=1$ for $|x-3|\leq 1$ and $\chi(x)=0$ for
$|x-3|\geq 2$. We set $u_\ell(p)=0$ for $p\in \mathcal M\setminus
(R,\infty)\times O_k$ and
\begin{equation}\label{test}
u_\ell(x,
y)=\varrho_k(x,y)\chi(x/\ell)\exp\bigl({ix(\xi+i\epsilon)}\bigr)\Phi(
y),\quad (x,y)\in (R,\infty)\times O_k,
\end{equation}
where $\Delta_{O_k}\Phi=\lambda \Phi$. Straightforward calculations
show that
\begin{equation*}
\|(\Delta_\epsilon-\mu^2\bigr)u_\ell; L^2_\epsilon(\mathcal M)\|\leq
C,\quad\|u_\ell; L^2(\mathcal M)\|\leq C,\quad
\|u_\ell;\mathcal D_\epsilon\|\to\infty
\end{equation*}
 as $\ell\to +\infty$. Thus the sequence $\{u_\ell\}$ violates the estimate~\eqref{peetre}
 and the operator~\eqref{ep} is not Fredholm.
\end{proof}

\begin{proof}[Proof of Lemma~\ref{l2}] As the result is essentially well-known, see e.g.~\cite[Chapter
5]{KMR} or~\cite{KM}, we only give a sketch of the proof. The notations below are the same as in the proof of Lemma~\ref{fr}.
Let  $u=(\Delta_0-\mu^2)^{-1}f$. Then  $\mathsf u_k=\varrho_k u\in
H^2_0(\Bbb R\times O_k)$, $1\leq k\leq n$, is a (unique)
solution to the equation~\eqref{cyl} with right hand side $\mathsf
f_k=\varrho_k f-[\Delta,\varrho_k]u$, where $\mathsf f_k$ is
extended to $\Bbb R\times O_k$ by zero. The inclusion $f\in L^2_\epsilon(\mathcal M)$ implies that the function $\xi\mapsto \hat{\mathsf f}_k(\xi)=\mathscr
F_{x\to \xi}\mathsf f_k\in L^2(O_k)$ is analytic in the strip
$|\Im \xi|< \epsilon$ with boundary values satisfying
$\int_{\Bbb R}\|\hat{\mathsf f}_k(\xi\pm i\epsilon);
L^2(O_k)\|^2\,d\xi<\infty$. We have
$$
\mathsf u_k=\mathscr F^{-1}_{\xi\to
x}(\Delta_{O_k}-\mu^2+\xi^2)^{-1}\hat{\mathsf f}_k(\xi).
$$
Let $\epsilon>0$ be less than the first positive eigenvalue of
$\Delta_{O_k}$. Then the resolvent $(\Delta_{O_k}-\mu^2+\xi^2)^{-1}$
is a meromorphic function of $\xi$ in the strip $-\epsilon\leq \Im
\xi\leq \epsilon$ having poles at $\xi=\pm\mu$, which
correspond to the zero eigenvalue and the constant eigenfunction of
$\Delta_{O_k}$. This together with the Cauchy's integral theorem
implies
$$
\mathsf u_k(x,y)=\mathsf v_k(x,y)+C_k
e^{\operatorname{sign}(\gamma_k)i \mu x},\quad \mathsf v_k=\mathscr
F^{-1}_{\xi\to
x}(\Delta_{O_k}-\mu^2+(\xi+i\gamma_k)^2)^{-1}\hat{\mathsf f}_k(\xi+i\gamma_k),
$$
where $\mathsf v_k$ is a unique in $H^2_\epsilon(\Bbb R\times O_k)$
solution to the equation~\eqref{cyl} and $C_k\in\Bbb C$ depends on
$\mu$ and $\mathsf f_k$. The term $C_k
e^{\operatorname{sign}(\gamma_k)i \mu x}=C_k
e^{i \mu x}$ (resp. $C_k
e^{\operatorname{sign}(\gamma_k)i \mu x}=C_k
e^{-i \mu x}$ ) appears   as the residue at the pole $\xi=\mu$  (resp. $\xi=-\mu$)  if $O_k$ corresponds to a right (resp. left) cylindrical end. As a consequence, for some
$c_k\in\Bbb C$ we have $\varrho_k u-c_k \varphi_k(\mu)\in \mathcal
D_\epsilon$, cf.~\eqref{phi}. Since $\varrho_0 u \in \mathcal
D_\epsilon$, we conclude that~\eqref{as} is valid provided
$$0<\epsilon<\lambda=\min_{\ell\in\Bbb
N,1\leq k\leq n} 4\pi^2 \ell^2 |O_k|^{-2},$$
 where $\lambda$
 is  the first positive eigenvalue of the selfadjoint
Laplacian on the union of $O_1,\dots,O_n$.
\end{proof}

 \subsection{Existence of embedded eigenvalues}

In this subsection we demonstrate that the selfadjoint Laplacian $\Delta$ on $(\mathcal M, m)$ can have eigenvalues embedded into the continuous spectrum $\sigma_c(\Delta)=[0,\infty)$. Let us  construct a simple suitable example of $(\mathcal M, m)$.

Consider the following strip with two semi-infinite slits:
$$
\Bbb S= \{x+iy\in\Bbb C: |x|\geq\pi, 0< |y-\pi/2|<\pi/2\}\cup \{x+iy\in\Bbb C: -\pi< x< \pi,0<y<\pi\}.
$$
Let $\Delta^D$ be the  Friedrichs selfadjoint extension of the Laplacian  $-\partial_x^2-\partial_y^2$ initially defined on the set $C_0^\infty(\Bbb S \setminus \{-\pi+i\pi/2, \pi+i\pi/2\})$. It is easy to check that $\Delta^D$ is positive and its continuous spectrum is $[4,\infty)$.
The first eigenvalue  of the Dirichlet Laplacian in the square $(-\pi/2,\pi/2)\times i(0,\pi)\subset \Bbb S$ is $2$. Extending the corresponding eigenfunction $\cos  x \sin  y $ to the strip $\Bbb S$ by zero, one obtains some function $u$ in the domain $H^1(\Bbb S)$ of the quadratic form $\mathsf q$ of $\Delta^D$. Clearly, $\mathsf q[u,u]=2$. Then the minimax principle implies that $\Delta^D$ has at least one (discrete) eigenvalue $\lambda\leq 2$ below the continuous spectrum $[4,\infty)$. We extend the corresponding eigenfunction $U$ to $\bar{\Bbb S}=\{x-iy: x+iy\in\Bbb S\}$ by setting $U(x,-y)=-U(x,y)$. Thus we constructed an eigenfunction $U$ corresponding to the (embedded) eigenvalue $\lambda\in(0,2]$ of the Laplacian $\Delta$ on the Mandelstam diagram  $(\mathcal M, m)$, where $\mathcal M$ is obtained from $\Bbb S\cup\bar {\Bbb S}$ by the following identifications of boundaries:
$$
\Bbb R+i\pi-i0 \text{ with } \Bbb R-i\pi+i0; \quad   \Bbb R+i0 \text{ with } \Bbb R-i0;
 $$
 $$
  \{x+i\pi/2+i0:|x|\geq\pi\} \text{ with } \{x-i\pi/2-i0:|x|\geq\pi\};
$$
$$
  \{x+i\pi/2-i0):|x|\geq\pi\} \text{ with } \{x-i\pi/2+i0:|x|\geq\pi\}.
$$

\end{document}